\documentclass{amsart}
\usepackage{amssymb,amscd}

\newtheorem{theorem}{Theorem}[section]
\newtheorem{lemma}[theorem]{Lemma}
\newtheorem{proposition}[theorem]{Proposition}
\newtheorem{corollary}[theorem]{Corollary}

\theoremstyle{definition}
\newtheorem{definition}[theorem]{Definition}

\newtheorem{examples}[theorem]{Examples}

\newtheorem{notation}[theorem]{Notation}

\theoremstyle{remark}
\newtheorem{remark}[theorem]{Remark}
\newtheorem{remarks}[theorem]{Remarks}

\numberwithin{equation}{section}

\newcommand{\spandsp}{\mbox{$\qquad\text{and}\qquad$}}
\newcommand{\sandss}{\mbox{$\quad\text{and}\;\;$}}

\newcommand{\sands}{\mbox{$\quad\text{and}\quad$}}
\newcommand{\letbe}{\mathbin{:\!\raisebox{-.36pt}{=}}\,}

\newcommand{\Ker}{\operatorname{Ker}}

\newcommand{\Tor}{\operatorname{Tor}}
\newcommand{\Ext}{\operatorname{Ext}}
\newcommand{\Ho}{\mathop\mathit{Ho}}

\newcommand{\holim}{\mathop\mathrm{holim}\nolimits}
\newcommand{\colim}{\mathop\mathrm{colim}\nolimits}
\newcommand{\hocolim}{\mathop\mathrm{hocolim}\nolimits}

\newcommand{\zk}{\mathcal Z_K}
\newcommand{\cl}{\mathop\mathit{CL}\nolimits}
\newcommand{\fl}{\mathop\mathit{FL}\nolimits}

\newcommand{\oline}[1]{\mbox{${\,\overline{\!#1}}$}}

\newcommand{\daja}{Davis-Januszkiewicz}
\newcommand{\daaja}{Davis and Januszkiewicz}

\newcommand{\djs}{\mbox{$D{\mskip-1mu}J\/$}}

\newcommand{\C}{\mathbb C}
\newcommand{\Q}{\mathbb Q}

\newcommand{\Z}{\mathbb Z}

\newcommand{\s}{\sigma}

\newcommand{\br}[1]{\mbox{$\langle #1\rangle$}}

\newcommand{\cat}[1]{\mbox{\sc #1}}
\newcommand{\scat}[1]{\mbox{\scriptsize{\sc #1}}}

\newcommand{\fcat}[2]{\mbox{\raisebox{1pt}{\rm\scriptsize [}{\sc #1},\hspace{1pt}{\sc #2}\raisebox{1pt}{\rm\scriptsize ]}}\hspace{1pt}}

\newcommand{\llongrightarrow}{\relbar\joinrel\longrightarrow}
\newcommand{\lllongrightarrow}{\relbar\joinrel\llongrightarrow}
\newcommand{\llllongrightarrow}{\relbar\joinrel\lllongrightarrow}
\newcommand{\lllllongrightarrow}{\relbar\joinrel\llllongrightarrow}
\newcommand{\llongleftarrow}{\longleftarrow\joinrel\relbar}
\newcommand{\lllongleftarrow}{\llongleftarrow\joinrel\relbar}

\newcommand{\lzze}{\stackrel{\simeq}{\longrightarrow}\dots\stackrel{\simeq}{\longleftarrow}}

\newcommand{\rla}{\begin{picture}(40,7.5)\put(6,-2.5){$\longleftarrow\joinrel\hspace{-1pt}\relbar\joinrel\hspace{-1pt}\relbar$}\put(6,2.5){$\relbar\joinrel\hspace{-1pt}\relbar\joinrel\hspace{-1pt}\longrightarrow$}
\end{picture}}

\newcommand{\under}{\!\downarrow\!}

\newcommand{\del}{\mbox{\footnotesize{$\Delta$}}}

\newcommand{\APL}{\mbox{$A_{P\!L}$}}

\begin{document}

\title{Categorical aspects of toric topology}

\author{Taras E Panov}
\address{\hspace{-\parindent}Department of Mathematics and Mechanics, Moscow State
University, Leninskie Gory, 119992 Moscow, Russia;\newline
\emph{and}\newline Institute for Theoretical and Experimental
Physics, Moscow 117259, Russia}

\email{tpanov@mech.math.msu.su}
\thanks{The first author was supported by EPSRC Visiting Fellowship
  GR/S55828/01 at the University of Manchester, by the Russian
  Foundation for Basic Research (grants 08-01-00541 and 08-01-91855-KO), and by
  P~Deligne's 2004 Balzan Prize for Mathematics}

\author{Nigel Ray} \address{\hspace{-\parindent}School of Mathematics, The University of
Manchester, Oxford Road, Manchester M13~9PL, England}
\email{nigel.ray@manchester.ac.uk}



\begin{abstract}
  We argue for the addition of category theory to the toolkit of toric
  topology, by surveying recent examples and applications. Our case is
  made in terms of toric spaces $X_K$, such as moment-angle complexes
  ${\mathcal Z}_K$, quasitoric manifolds $M$, and Davis-Januszkiewicz
  spaces $D\hspace{-1pt}J(K)$. We first exhibit $X_K$ as the homotopy
  colimit of a diagram of spaces over the small category {\sc
    cat}$(K)$, whose objects are the faces of a finite simplicial
  complex $K$ and morphisms their inclusions. Then we study the
  corresponding {\sc cat}$(K)$-diagrams in various algebraic Quillen
  model categories, and interpret their homotopy colimits as algebraic
  models for $X_K$. Such models encode many standard algebraic
  invariants, and their existence is assured by the Quillen structure.
  We provide several illustrative calculations, often over the
  rationals, including proofs that quasitoric manifolds (and various
  generalisations) are rationally formal; that the rational Pontrjagin
  ring of the loop space $\varOmega D\hspace{-1pt}J(K)$ is isomorphic
  to the quadratic dual of the Stanley-Reisner algebra $\mathbb Q[K]$ for flag
  complexes $K$; and that $D\hspace{-1pt}J(K)$ is coformal precisely
  when $K$ is flag. We conclude by describing algebraic models for the
  loop space $\varOmega D\hspace{-1pt}J(K)$ for any complex $K$, which
  mimic our previous description as a homotopy colimit of topological
  monoids.
\end{abstract}

\maketitle

%
%
%
%
%
%
%
%
%

\section{Introduction}\label{in}

\subsection{Toric Topology}
Toric topology is rapidly gaining recognition as an area of
independent mathematical interest, and the aim of this article is to
survey aspects of category theory that are exerting a growing
influence on its development. The primary objects of study are derived
from well-behaved actions of the $n$-dimensional torus $T^n$ on a
topological space, and lie in a variety of geometric and algebraic
categories. We therefore refer to them as \emph{toric objects}. Each
of the orbit spaces is equipped with a natural combinatorial
structure, which encodes the distribution of isotropy subgroups and is
determined by a finite simplicial complex $K$. In this context, we
wish to describe and evaluate topological and homotopy theoretic
invariants of toric spaces in terms of combinatorial data associated
to $K$.

A fundamental example is provided by the moment-angle complex $\zk$
\cite{bu-pa02}, which depends only on the choice of $K$. When $K$ has
$m$ vertices, $\zk$ supports a canonical action of $T^m$ whose
quotient is naturally homeomorphic to the cone $CK'$ on the
barycentric subdivision of $K$. If $K$ is the boundary of an
$n$--dimensional simplicial polytope then $CK'$ may be identified with
its dual polytope $P$, which is simple. In this case, certain subtori
$T^{m-n}<T^m$ act freely on $\zk$, and their quotient spaces $M^n$ are
the toric manifolds introduced by \daaja\ \cite{da-ja91}. In order to
avoid confusion with the toric varieties of algebraic geometry, we
follow current convention \cite{bu-pa02} by labelling them
\emph{quasitoric manifolds}. The quotient of any such $M$ by the
$n$--torus $T^m/T^{m-n}$ is naturally homeomorphic to $P$.

It is convenient to describe several families of toric spaces
$X_K$ as colimits of diagrams of topological spaces over the small
category $\cat{cat}(K)$, whose objects are the faces of $K$ and
morphisms their inclusions. In certain cases, algebraic invariants
of $X_K$ may then be described in terms of the corresponding
diagrams in appropriate algebraic categories; in other words, by
algebraic toric objects. A typical example is given by the \daja\
space $\djs(K)$. As introduced in \cite{da-ja91}, $\djs(K)$ is
defined by the Borel construction $\zk\times_{T^m}ET^m$, and its
cohomology ring $H^*(\djs(K);R)$ is isomorphic to the
Stanley-Reisner algebra $R[K]$ \cite{stan96} for any coefficient
ring $R$. Following \cite{bu-pa02} and \cite{p-r-v04}, it may be
interpreted as the colimit of a diagram $BT^K$, whose value on
each face $\sigma$ of $K$ is the cartesian product $BT^\sigma$;
then $R[K]$ is isomorphic to the limit of the corresponding
$\cat{cat}^{op}(K)$-diagram of polynomial algebras $S_R(K)$, whose
value on $\sigma$ is the polynomial algebra $S_R(\sigma)$
generated by its vertices.

Such constructions involve the categories $\cat{cat}(K)$ in a local
r\^ole, and are already well-publicised. We believe that category
theory also offers an important global viewpoint, which suggests a
systematic programme for the evaluation of algebraic and geometrical
invariants of $X_K$.

\subsection{Categorical Motivation}
The motivation for our programme lies in two closely related
observations about the local situation. One is that many familiar
invariants of toric spaces depend only on their homotopy type, yet
homotopy equivalences do not interact well with colimits; the other is
that many of the functors we wish to apply to $X_K$ do not respect
colimits, and may therefore be difficult to evaluate. If, however, we
can justify interpreting $X_K$ as a homotopy colimit \cite{bo-ka72},
such difficulties often evaporate. This possibility seems first to
have been recognised by Welker, Ziegler, and \v{Z}ivaljevi\'{c}
\cite{w-z-z99}, who show that toric varieties themselves are
expressible as homotopy colimits. Our own evidence is provided in
\cite{p-r-v04}, where the loop space $\varOmega\djs(K)$ is modelled
geometrically by a homotopy colimit in the category $\cat{tmon}$ of
topological monoids. Developing these themes leads us naturally into
the world of Quillen model categories \cite{quil67}.

We therefore work with model categories whenever we are able. As well
as being currently fashionable, Quillen's theory suggests questions
that we might not otherwise have asked, and presents challenges of
independent interest. For example, the existence of homotopy colimits
in an arbitrary model category has been known for some time
\cite{hirs03}, but specific constructions are still under development
\cite{walt05} in several cases that we discuss below; and the advent
of \cite{d-h-k-s04} raises the possibility of working in the more
general framework of \emph{homotopical categories}.

We consider images of an arbitrary toric space $X_K$ under various
functors taking values in algebraic and geometric model categories,
and refer to the resulting toric objects as {\it algebraic\/} or {\it
  geometric models\/} for $X_K$. When our models are covariant, we
prove that they are equivalent to the homotopy colimits of the
corresponding diagrams; and when contravariant, to their homotopy
limits.

An influential example is provided by Bousfield and Gugenheim's
treatment \cite{bo-gu76} of the model category $\cat{cdga}$ of
commutative differential graded algebras over the rationals $\Q$.
Their work combines with that of Sullivan \cite{sull78} to show that
problems of rational homotopy theory for nilpotent spaces of finite
type may be solved in the homotopy category $\Ho(\cat{cdga})$ by
applying the PL-cochain functor $\APL(-)$. Many toric spaces are
simply connected CW-complexes of finite type, and their cellular
cochain complexes have trivial differential. So their rationalisations
retain much of the original homotopy theoretic information, and have
interesting geometrical properties of their own.  These may then be
deduced from the contravariant algebraic models $\APL(X_K)$, which we
express as homotopy limits in $\cat{cdga}$. We also consider other
rational models of similar power, such as differential graded
coalgebras, Lie algebras, and Pontrjagin rings.

The categorical viewpoint has already motivated studies such as
\cite{no-ra05}, where the rational formality of $\djs(K)$ is
established for every simplicial complex $K$. We extend this result
below, to a class of toric spaces that includes quasitoric manifolds
and the torus manifolds of \cite{ma-pa06} as special cases. By way of
contrast, we note that calculations of Baskakov~\cite{bask03} and
Denham and Suciu \cite{de-su06} confirm that many moment-angle
complexes $\zk$ support non-trivial Massey products, and so cannot be
formal. A further goal is to place these facts in the context of
\cite{p-r-v04}, where the properties of $\cat{tmon}$ as a geometric
model category play an important r\^ole, but remain implicit. We
therefore study the rational \emph{\it co}formality \cite{ne-mi78} of
$\djs(K)$, which depends on the rational structure of
$\varOmega\djs(K)$ and is verified for any flag complex $K$. Our study
includes an investigation of the Lie algebra
$\pi_*(\varOmega\djs(K))\otimes_{\Z}\Q$, and is related to
calculations of the Pontrjagin ring $H_*(\varOmega\djs(K);\Q)$. It is
particularly fascinating to try and understand how the combinatorial
structure of $K$ influences the commutators and higher commutators
which characterise the latter; our initial calculations in this
direction have already been supplemented by those of \cite{de-su06}.

\subsection{Contents}
The contents of our sections are as follows.

In Section \ref{mocadi} we recall background information on
general category theory, introducing notation as we proceed. We
place special emphasis on the finite categories $\cat{cat}(K)$,
aspects of Quillen's model category theory, and categories of
$\cat{cat}(K)$-diagrams. We introduce our algebraic model
categories in Section \ref{almoca}, where we emphasise less
familiar cases by describing explicit fibrations and cofibrations.
These ideas underlie Section \ref{holico}, where we outline the
construction and relevant properties of homotopy limits and
colimits in terms of fibrant and cofibrant replacement functors.
The remaining sections are focused on applications, and Section
\ref{modjs} begins with straightforward examples arising from the
Stanley-Reisner algebra of a simplical complex $K$ and the related
space $\djs(K)$. Similar examples follow in Section \ref{mozk},
where we discuss algebraic models for moment-angle complexes, and
in Section \ref{moqtm}, where we introduce algebraic models for
quasi-toric manifolds and confirm that they are rationally formal.
We then progress to loop spaces, and in Section \ref{mlsak} we
establish algebraic analogues of our geometric model
\cite{p-r-v04} for $\varOmega\djs(K)$ as a homotopy colimit of
topological monoids. We specialise to flag complexes $K$ in
Section \ref{mlsfk}, by proving that the rational Pontrjagin ring
of $\varOmega\djs(K)$ is isomorphic to the quadratic dual of the
Stanley-Reisner algebra $\Q[K]$ and that $\djs(K)$ is coformal. In
our concluding Section \ref{examp} we give additional examples
concerning the cobar construction and higher commutators in
Pontrjagin rings.

\subsection{Algebraic Conventions}
So far as general algebraic notation is concerned, we work over an
arbitrary commutative ring $R$, usually indicated by means of a
subscript. In many situations $R$ is restricted to the rational
numbers, and in that case only we omit the subscript for reasons of
notational clarity.

We consider finite sets $W$ of generators $w_1$, \dots, $w_m$, which
are graded over the non-negative integers by a dimension function
$|w_j|$, for $1\leq j\leq m$. We write the graded tensor $R$-algebra
on $W$ as $T_R(w_1,\dots,w_m)$, and use the abbreviation $T_R(W)$
whenever possible. Its symmetrisation $S_R(W)$ is the graded
commutative $R$-algebra generated by $W$. If $U$, $V\subseteq W$ are
the subsets of odd and even grading respectively, then $S_R(W)$ is the
tensor product of the exterior algebra $\wedge_R(U)$ and the
polynomial algebra $P_R(V)$. It is also convenient to denote the free
graded Lie algebra on $W$ and its commutative counterpart by
$\fl_R(W)$ and $\cl_R(W)$ respectively; the latter is nothing more
than a free $R$-module.

We adapt this notation to subsets $\alpha\subseteq W$ by writing
$T_R(\alpha)$, $S_R(\alpha)$, $\wedge_R(\alpha)$, $\fl_R(\alpha)$,
and $\cl_R(\alpha)$ as appropriate. Identifying $\alpha$ with its
characteristic function then allows us to denote the square-free
monomial $\prod_\alpha w_i$ by $w_\alpha$ in $S_R(\alpha)\leq
S_R(W)$.

Almost all of our graded algebras have finite type, leading to a
natural coalgebraic structure on their duals. We write the free tensor
coalgebra on $W$ as $T_R\br{W}$;
it is isomorphic
to $T_R(W)$ as $R$-modules, and its diagonal is given by
\[
\delta(w_{j_1}\otimes\dots\otimes w_{j_r})=
\sum_{k=0}^r(w_{j_1}\otimes\dots\otimes w_{j_k})
\otimes(w_{j_{k+1}}\otimes\dots\otimes w_{j_r}).
\]
The submodule $S_R\br{W}$
of symmetric elements ($w_i\otimes
w_j+(-1)^{|w_i||w_j|}w_j\otimes w_i$, for example) is the graded
cocommutative $R$-coalgebra cogenerated by $W$.

Given $W$, we may sometimes define a differential by denoting the
set of elements $dw_1$, \dots, $dw_m$ by $dW$. The differential
lowers or raises gradings by $1$, depending on whether it is
homological or cohomological respectively. For example, we write
the free differential graded algebra on a single generator $w$ of
positive dimension as $T_R(w,dw)$; the notation is designed to
reinforce the fact that its underlying algebra is the tensor
$R$-algebra on elements $w$ and $dw$.  Similarly, $T_R\br{w,dw}$
is the free differential graded coalgebra on $w$. For further
information on differential graded coalgebras, \cite{h-m-s74}
remains a valuable source.

\subsection{Acknowledgements}
Since 2001, we have benefitted greatly from the advice and
encouragement of many colleagues, particularly John Greenlees,
Dietrich Notbohm, Brooke Shipley, and most notably, Rainer Vogt.
The second-named author is especially grateful to the organisers of
the
Osaka conference for the opportunity to present our work there; he is
also responsible for several long delays in completing the latex file,
and offers his apologies to all those colleagues who were promised a
final version in 2002.

%
%
%
%
%
%
%
%
%

\section{Diagrams and model categories}\label{mocadi}

In this section we introduce aspects of category theory that are
directly relevant to the study of toric spaces. We begin by recalling
the finite category $\cat{cat}(K)$, associated to an arbitrary
simplicial complex $K$. For global purposes we turn to the notion of a
Quillen model category, and outline its relevance to rational homotopy
theory. The two are interwoven in the study of categories of diagrams,
whose terminology and notation we introduce as we proceed. For more
complete background information we refer to the books of Hirschhorn
\cite{hirs03} and Hovey \cite{hove99}.

We define our simplicial complexes $K$ on a graded set $V$ of {\it
vertices\/} $v_1$, \dots, $v_m$, each of which has dimension $2$. So
$K$ is a collection of {\it faces} $\sigma\subseteq V$, closed under
the formation of subsets and including the empty face $\varnothing$.
Every face determines two particular subcomplexes of $K$, namely the
simplex $\varDelta(\sigma)$, and its boundary $\partial(\sigma)$; the
former consists of all faces $\tau\subseteq\sigma$, and the latter is
obtained by omitting the maximal face $\sigma$. For algebraic
purposes, we occasionally prefer the vertices to have grading $1$, in
which case we may replace $V$ by $U$ and $v_i$ by $u_i$ for emphasis.

We shall work with the following combinatorial categories, noting that
$\del$ is small, and $\cat{cat}(K)$ is small and finite.
\begin{notation}\hfill
\begin{itemize}
\item $\cat{set}$: sets and functions;
\item $\del$: finite ordinals and nondecreasing functions;
\item $\cat{cat}(K)$: faces of a finite simplicial complex $K$, and
their inclusions.
\end{itemize}
\end{notation}

Given a small category $\cat{a}$ and an arbitrary category $\cat{r}$,
a covariant functor $D\colon\cat{a}\to\cat{r}$ is known as an
\emph{$\cat{a}$-diagram} in $\cat{r}$. Such diagrams are themselves
the objects of a \emph{diagram category\/} $\fcat{a}{r}$, whose
morphisms are natural transformations. When $\cat{a}$ is $\del^{op}$,
the diagrams are precisely the simplicial objects in $\cat{r}$, and
are written as $D_\bullet$; the object $D(n)$ is abbreviated to $D_n$
for every $n\geq 0$, and forms the $n$--simplices of
$D_\bullet$. Motivated by the example $\cat{sset}$ of simplicial sets,
we may abbreviate the diagram category to $\cat{sr}$ in this case
only.

We may interpret every object $r$ of $\cat{r}$ as a \emph{constant
  $\cat{a}$-diagram}, and so define the constant functor
$\kappa\colon\cat{r}\rightarrow\fcat{a}{r}$. Whenever $\kappa$ admits
a right or left adjoint $\fcat{a}{r}\rightarrow\cat{r}$, it is known
as the \emph{limit} or \emph{colimit} functor respectively.

For any object $r$ of $\cat{r}$, the objects of the
\emph{overcategory} $\cat{r}\under r$ are morphisms $f\colon
q\rightarrow r$, and the morphisms are the corresponding commutative
triangles; the full subcategory $\cat{r}\!\Downarrow\!r$ is given by
restricting attention to non-identities $f$. Similarly, the objects of
the \emph{undercategory} $r\under\cat{r}$ are morphisms
$f\colon r\rightarrow s$, and the morphisms are the corresponding
triangles; $r\!\Downarrow\!\cat{r}$ is given by restriction to the
non-identities. In $\cat{cat}(K)$ for example, we have that
\[
\cat{cat}(K)\under\sigma=\cat{cat}(\varDelta(\sigma))\sands
\cat{cat}(K)\!\Downarrow\!\sigma=\cat{cat}(\partial(\sigma))
\]
for any face $\sigma$. As usual, we write $\cat{r}(r,s)$ for the set
of morphisms $r\to s$ in $\cat{r}$.

A \emph{model category} $\cat{mc}$ is closed with respect to the
formation of certain limits and colimits, and contains three
distinguished subcategories, whose morphisms are \emph{weak
equivalencies} $e$, \emph{fibrations} $f$, and \emph{cofibrations} $g$
respectively. Unless otherwise stated, these letters denote such
morphisms henceforth. A fibration or cofibration is \emph{acyclic}
whenever it is also a weak equivalence. The three subcategories
satisfy certain axioms, for which we follow Hirschhorn
\cite[Definition 7.1.3]{hirs03}; these strengthen Quillen's original
axioms for a closed model category \cite{quil67} in two minor but
significant ways. Firstly, we demand closure with respect to
\emph{small} limits and colimits, whereas Quillen insists only that
they be finite. Secondly, we demand that every morphism $h$ should
factorise \emph{functorially} as
\begin{equation}\label{mcfacts}
h\;=\;f\cdot g\;=\;f'\cdot g'
\end{equation}
for some acyclic $f$ and $g'$, whereas Quillen insists only that such
factorisations exist. When using results of pioneering authors such as
Bousfield and Gugenheim \cite{bo-gu76} and Quillen \cite{quil69}, we
must take account of these differences.

The axioms imply that initial and terminal objects $\circ$ and $*$
exist in $\cat{mc}$, and that $\cat{mc}\!\downarrow\!M$ and
$M\!\downarrow\!\cat{mc}$ inherit model structures for any object $M$.

An object of $\cat{mc}$ is \emph{cofibrant} when the natural morphism
$\circ\to M$ is a cofibration, and is \emph{fibrant} when the natural
morphism $M\to *$ is a fibration. A \emph{cofibrant approximation} to
an object $N$ is a weak equivalence $N'\to N$ with cofibrant source,
and a \emph{fibrant approximation} is a weak equivalence $N\to N''$
with fibrant target. The full subcategories $\cat{mc}_c$,
$\cat{mc}_{\!f}$ and $\cat{mc}_{cf}$ are defined by restricting
attention to those objects of $\cat{mc}$ that are respectively
cofibrant, fibrant, and both. When applied to $\circ\to N$ and $N\to
*$, the factorisations \eqref{mcfacts} determine a \emph{cofibrant
replacement} functor $\omega\colon\cat{mc}\to\cat{mc}_c$, and a
\emph{fibrant replacement} functor
$\phi\colon\cat{mc}\to\cat{mc}_{\!f}$. It follows from the definitions
that $\omega$ and $\phi$ preserve weak equivalences, and that the
associated acyclic fibrations $\omega(N)\to N$ and acyclic
cofibrations $N\to \phi(N)$ form cofibrant and fibrant approximations
respectively.  These ideas are central to our definition of
homotopy limits and colimits, and we shall see many examples below.

Weak equivalences need not be invertible, so objects $M$ and $N$ are
deemed to be \emph{weakly equivalent} if they are linked by a zig-zag
$M\stackrel{e_1}{\longrightarrow}\dots\stackrel{e_n}{\longleftarrow}N$
in $\cat{mc}$; this is the smallest equivalence relation generated by
the weak equivalences. An important consequence of the axioms is the
existence of a localisation functor
$\gamma\colon\cat{mc}\to\Ho(\cat{mc})$, such that $\gamma(e)$ is an
isomorphism in the \emph{homotopy category} $\Ho(\cat{mc})$ for every
weak equivalence. Here $\Ho(\cat{mc})$ has the same objects as
$\cat{mc}$, and is equivalent to a category whose objects are those of
$\cat{mc}_{cf}$, but whose morphisms are homotopy classes of morphisms
between them. In $\cat{mc}_{cf}$, homotopy is an equivalence relation
defined by means of \emph{cylinder} or \emph{path objects}.

Any functor $F$ of model categories that preserves weak equivalences
necessarily induces $\Ho(F)$ on their homotopy categories, although
weaker conditions suffice. Examples of the former include
\begin{equation}\label{hoomfi}
\Ho(\omega)\colon\Ho(\cat{mc})\to\Ho(\cat{mc}_c)\sands
\Ho(\phi)\colon\Ho(\cat{mc})\to\Ho(\cat{mc}_{\!f}).
\end{equation}
Such functors often occur as adjoint pairs
\begin{equation}\label{adjpair}
F\colon\cat{mb}\rla\cat{mc}:\! G\;,
\end{equation}
where $F$ is \emph{left Quillen} if it preserves cofibrations and
acyclic cofibrations, and $G$ is \emph{right Quillen} if it preserves
fibrations and acyclic fibrations. Either of these implies the other,
leading to the notion of a \emph{Quillen pair} $(F,G)$; then Ken
Brown's Lemma \cite[Lemma 7.7.1]{hirs03} applies to show that $F$ and
$G$ preserve all weak equivalences on $\cat{mb}_c$ and
$\cat{mc}_{\!f}$ respectively. So they may be combined with
\eqref{hoomfi} to produce an adjoint pair of \emph{derived functors}
\[
LF\colon\Ho(\cat{mb})\rla\Ho(\cat{mc}):\! RG,
\]
which are equivalences of the homotopy categories (or certain of their
full subcategories) in favourable cases.

Our first examples of model categories are geometric, as
follows.
\begin{notation}\hfill
\begin{itemize}
\item $\cat{top}$: pointed $k$-spaces and continuous maps
 \cite{vogt71};
\item
$\cat{tmon}$: topological monoids and continuous homomorphisms.
\end{itemize}
\end{notation}
We assume that topological monoids are $k$-spaces and are pointed by
their identities, so that 
$\cat{tmon}$ is a subcategory of $\cat{top}$. The standard model
structure for $\cat{top}$ is described in detail by Hovey
\cite[Theorem 2.4.23]{hove99}; weak equivalences induce isomorphisms
of homotopy groups, fibrations are Serre fibrations, and cofibrations
obey the left lifting property with respect to acyclic fibrations. The
model structure for $\cat{tmon}$ is originally due to Schw\"{a}nzl and
Vogt \cite{sc-vo91}, and may also be deduced from Schwede and
Shipley's theory \cite{sc-sh00} of monoids in monoidal model
categories; weak equivalences and fibrations are those homomorphisms
which are weak equivalences and fibrations in \cat{top}, and
cofibrations obey the appropriate lifting property.

Our algebraic categories are defined over arbitrary commutative rings
$R$, but tend only to acquire model structures when $R$ is a field of
characteristic zero. If $R=\Q$, and in this case only, we omit the
subscript from the notation.
\begin{notation}\label{algcats}\hfill
\begin{itemize}
\item $\cat{ch}_R$ and $\cat{coch}_R$: augmented chain and cochain
complexes;
\item $\cat{cdga}_R$: commutative augmented differential graded
algebras, with cohomology differential;
\item
$\cat{dga}_R$: augmented differential graded algebras, with homology
differential;
\item
$\cat{cdgc}_R$: cocommutative supplemented differential graded
coalgebras, with homology differential;
\item
$\cat{dgc}_R$: supplemented differential graded coalgebras, with homology
differential;
\item
$\cat{dgl}$: differential graded Lie algebras over $\Q$, with homology
differential.
\end{itemize}
\end{notation}
For any model structure on these categories, weak equivalences are the
{\it quasi-isomorphisms}, which induce isomorphisms in homology or
cohomology. The fibrations and cofibrations are described in Section
\ref{almoca} below. The augmentations and supplementations act as
algebraic analogues of basepoints.

We reserve the notation $\cat{amc}$ for any of the categories
\ref{algcats}, and assume that objects are graded over the
non-negative integers, and therefore \emph{connective}; for $i\geq 0$,
we denote the full subcategory of $i$-connected objects by
$\cat{amc}_i$. In order to emphasise the differential, we may
sometimes display an object $M$ as $(M,d)$. The homology or cohomology
group $H(M,d)$ is also an $R$-module, and inherits all structure on
$M$ except for the differential. Nevertheless, we may interpret any
graded algebra, coalgebra or Lie algebra as an object of the
corresponding differential category, by imposing $d=0$.
\begin{definition}\label{focofo}
An object $(M,d)$ is {\it formal in\/} $\cat{amc}$ whenever there exists a
zig-zag of quasi-isomorphisms
\begin{equation}\label{zigzag}
  (M,d)=M_1\lzze M_k=(H(M),0).
\end{equation}
\end{definition}
Formality only has meaning in an algebraic model category.

Sullivan's approach to rational homotopy theory is based on the
PL-cochain functor $\APL\colon\cat{top}\rightarrow\cat{cdga}$.
Following \cite{f-h-t01}, $\APL(X)$ is defined as $A^*(S_\bullet X)$,
where $S_\bullet(X)$ denotes the total singular complex of $X$ and
$A^*\colon\cat{sset}\rightarrow\cat{cdga}$ is the polynomial de Rham
functor of \cite{bo-gu76}. The PL-de Rham Theorem yields a natural
isomorphism $H(\APL(X))\rightarrow H^*(X,\Q)$, so $\APL(X)$ provides a
commutative replacement for rational singular cochains, and $\APL$
descends to homotopy categories. Bousfield and Gugenheim describe its
derived functor in terms of minimal models, and prove that it
restricts to an equivalence of appropriate full subcategories of
$\Ho(\cat{top})$ and $\Ho(\cat{cdga})$. In other words, it provides a
contravariant algebraic model for the rational homotopy theory of
well-behaved spaces.

Quillen's approach involves the homotopy groups $\pi_*(\varOmega
X)\otimes_\Z\Q$, which form the \emph{rational homotopy Lie algebra of
$X$} under the \emph{Samelson product}. He constructs a covariant
functor $Q\colon\cat{top}_1\rightarrow\cat{dgl}_0$, and a natural
isomorphism
\[
H(Q(X))\stackrel{\cong}{\longrightarrow}\pi_*(\varOmega
X)\otimes_{\Z}\Q.
\]
for any simply connected $X$. He concludes that $Q$ passes to an
equivalence of homotopy categories; in other words, its derived
functor provides a covariant algebraic model for the rational homotopy
theory of simply connected spaces.

The two approaches are Eckmann-Hilton dual, but the details are
subtle. Each has enabled important calculations, leading to the
solution of significant geometric problems. For examples, and
further details, we refer readers to \cite{f-h-t01}.

Definition \ref{focofo} is consistent with standard terminology, which
describes a topological space $X$ as \emph{formal} when $\APL(X)$ is
formal in $\cat{cdga}$ \cite{sull78}, and \emph{coformal} when $Q(X)$
is formal in $\cat{dgl}$ \cite{ne-mi78}.  In particular, $X$ is formal
whenever there exists a geometric procedure for making a
multiplicative choice of cocycle to represent each cohomology class;
this yields a quasi-isomorphism $H^*(X,\Q)\to\APL(X)$, and applies to
spaces such as $\C P^\infty$.

The importance of categories of simplicial objects is due in part to
the structure of the indexing category $\del^{op}$. Every object $(n)$
has \emph{degree} $n$, and every morphism may be factored uniquely as
a composition of morphisms that raise and lower degree. These
properties are formalised in the notion of a \emph{Reedy category}
$\cat{a}$, which admits generating subcategories $\cat{a}_+$ and
$\cat{a}_-$ whose non-identity morphisms raise and lower degree
respectively.  The diagram category $\fcat{a}{mc}$ then supports a
canonical model structure of its own \cite[Theorem 15.3.4]{hirs03}. By
duality, $\cat{a}^{op}$ is also Reedy, with
$(\cat{a}^{op})_+=(\cat{a}_-)^{op}$ and vice-versa. A simple example
is provided by $\cat{cat}(K)$, whose degree function assigns the
dimension $|\sigma|-1$ to each face $\sigma$ of $K$. So
$\cat{cat}_+(K)$ is the same as $\cat{cat}(K)$, and $\cat{cat}_-(K)$
consists entirely of identities.

In the Reedy model structure on $\fcat{cat$(K)$}{mc}$, weak
equivalences $e\colon C\rightarrow D$ are given \emph{objectwise}, in
the sense that $e(\sigma)\colon C(\sigma)\rightarrow D(\sigma)$ is a
weak equivalence in $\cat{mc}$ for every face $\sigma$ of
$K$. Fibrations are also objectwise. To describe the cofibrations, we
restrict $C$ and $D$ to the overcategories
$\cat{cat}(\partial(\sigma))$, and write $L_\sigma C$ and $L_\sigma D$
for their respective colimits. So $L_\sigma$ is the \emph{latching
functor} of \cite{hove99}, and $g\colon C\rightarrow D$ is a
cofibration precisely when the induced maps
\begin{equation}\label{diagcof}
C(\sigma)\amalg_{L_\sigma C}L_\sigma D\longrightarrow D(\sigma)
\end{equation}
are cofibrations in $\cat{mc}$ for all faces $\sigma$. Thus
$D\colon\cat{cat}(K)\to\cat{mc}$ is cofibrant when every canonical map
$\colim D(\partial(\sigma))\rightarrow D(\sigma)$ is a cofibration.

In the dual model structure on $\fcat{cat$^{op}(K)$}{mc}$, weak
equivalences and cofibrations are given objectwise. To describe the
fibrations, we restrict $C$ and $D$ to the undercategories
$\cat{cat}^{op}(\partial(\sigma))$, and write $M_\sigma C$ and
$M_\sigma D$ for their respective limits. So $M_\sigma$ is the
\emph{matching functor} of \cite{hove99}, and $f\colon C\rightarrow D$
is a fibration precisely when the induced maps
\begin{equation}\label{diagf}
C(\sigma)\longrightarrow D(\sigma)\times_{M_\sigma D}M_\sigma C
\end{equation}
are fibrations in $\cat{mc}$ for all faces $\sigma$. Thus
$C\colon\cat{cat}^{op}(K)\to\cat{mc}$ is fibrant when every canonical
map $C(\sigma)\rightarrow\lim C(\partial(\sigma))$ is a fibration.

The axioms for a model category are actually self-dual, in the sense
that any general statement concerning fibrations, cofibrations,
limits, and colimits is equivalent to the statement in which they are
replaced by cofibrations, fibrations, colimits, and limits
respectively. In particular, $\cat{mc}^{op}$ always admits a dual
model structure.

%
%
%
%
%
%
%
%
%

\section{Algebraic model categories}\label{almoca}

In this section we give further details of the algebraic model
categories of \ref{algcats}. We describe the fibrations and
cofibrations in each category, comment on the status of the
strengthened axioms, and give simple examples in less familiar cases.
We also discuss two important adjoint pairs.

\subsection{Chain and cochain complexes}\label{chcoch}
The existence of a model structure on categories of chain complexes
was first proposed by Quillen \cite{quil67}, whose view of homological
algebra as homotopy theory in $\cat{ch}_R$ was a crucial insight.
Variations involving bounded and unbounded complexes are studied by
Hovey \cite{hove99}, for example. In $\cat{ch}_R$, we assume that the
fibrations are epimorphic in positive degrees and the cofibrations are
monomorphic with degree-wise projective cokernel \cite{dw-sp95}. In
particular, every object is fibrant.

The existence of limits and colimits is assured by working
dimensionwise, and functoriality of the factorisations \eqref{mcfacts}
follows automatically from the fact that $\cat{ch}_R$ is
\emph{cofibrantly generated} \cite[Chapter 2]{hove99}.

Tensor product of chain complexes invests $\cat{ch}_R$ with the
structure of a monoidal model category, as defined by Schwede and
Shipley \cite{sc-sh00}. The Dold-Kan correspondence confirms that the
normalisation functor $N\colon\cat{smod}_R\rightarrow\cat{ch}_R$ is an
equivalence of categories \cite[\S3]{go-ja99},
and one of a Quillen pair.
Being a \emph{simplicial model category} \cite{hirs03}, $\cat{smod}_R$
is more amenable to various homotopy theoretic constructions.

Model structures on $\cat{coch}_R$ are established by analogous
techniques. It is usual to assume that the fibrations are epimorphic
with degree-wise injective kernel, and the cofibrations are
monomorphic in positive degrees. Then every object is cofibrant. There
is an alternative structure based on projectives, but we shall only
refer to the rational case so we ignore the distinction. Tensor
product turns $\cat{coch}_R$ into a monoidal model category.

\subsection{Commutative differential graded algebras}\label{cdga}
We consider commutative differential graded algebras over $\Q$ with
cohomology differentials, so they are commutative monoids in
$\cat{coch}$. A model structure on $\cat{cdga}$ was first defined in
this context by Bousfield and Gugenheim \cite{bo-gu76}, and has played
a significant role in the theoretical development of rational homotopy
theory ever since. The fibrations are epimorphic, and the cofibrations
are determined by the appropriate lifting property; some care is
required to identify sufficiently many explicit cofibrations.

Limits in $\cat{cdga}$ are created in the underlying category
$\cat{coch}$ and endowed with the natural algebra structure, whereas
colimits exist because $\cat{cdga}$ has finite coproducts and filtered
colimits. The proof of the factorisation axioms in \cite{bo-gu76} is
already functorial.

By way of example, we note that the product of algebras $A$ and $B$ is
their augmented sum $A\oplus B$, defined by pulling back the diagram
of augmentations
\[
A\stackrel{\varepsilon_A}{\lllongrightarrow}\Q
\stackrel{\varepsilon_B}{\lllongleftarrow}B
\]
in $\cat{coch}$ and imposing the standard multiplication on the
result. The coproduct is their tensor product $A\otimes B$ over
$\Q$. Examples of cofibrations include extensions of the form
$A\to(A\otimes S(w),d)$, determined by cocycles $z$ in $A$; such
an extension is defined by pushing out the diagram
\begin{equation}\label{frexcdga}
A\stackrel{h}{\llongleftarrow}S(x)\stackrel{j}\llongrightarrow
S(w,dw),
\end{equation}
where $h(x)=z$ and $j(x)=dw$. So the differential on $A\otimes S(w)$
is given by
\[
  d(a\otimes 1)=d_Aa\otimes 1\sandss d(1\otimes w)=z\otimes 1.
\]
This illustrates the fact that the pushout of a cofibration is a
cofibration. A larger class of cofibrations $A\rightarrow A\otimes
S(W)$ is given by iteration, for any set $W$ of positive
dimensional generators corresponding to cocycles in $A$.

The factorisations \eqref{mcfacts} are only valid over fields of
characteristic $0$, so the model structure does not extend to
$\cat{cdga}_R$ for arbitrary rings $R$.

\subsection{Differential graded algebras}\label{dga}
Our differential graded algebras have homology differentials, and are
the monoids in $\cat{ch}_R$. A model category structure in
$\cat{dga}_R$ is therefore induced by applying Quillen's path object
argument, as in \cite{sc-sh00}; a similar structure was first proposed
by Jardine, \cite{jard97} (albeit with cohomology differentials), who
proceeds by modifying the methods of \cite{bo-gu76}. Fibrations are
epimorphisms, and cofibrations are determined by the appropriate
lifting property.

Limits are created in $\cat{ch}_R$, whereas colimits exist because
$\cat{dga}_R$ has finite coproducts and filtered colimits.
Functoriality of the factorisations follows by adapting the proofs of
\cite{bo-gu76}, and works over arbitrary $R$.

For example, the coproduct of algebras $A$ and $B$ is the free
product \mbox{$A\star B$}, formed by factoring out an appropriate
differential graded ideal \cite{jard97} from the free tensor
algebra $T_R(A\otimes B)$ on the chain complex $A\otimes B$.
Examples of cofibrations include the extensions $A\to(A\star
T_R(w),d)$, determined by cycles $z$ in $A$. By analogy with the
commutative case, such an extension is defined by pushing out the
diagram
\begin{equation}\label{frexdga}
A\stackrel{h}{\llongleftarrow}T_R(x)\stackrel{j}\llongrightarrow
T_R(w,dw),
\end{equation}
where $h(x)=z$ and $j(x)=dw$. The differential on $A\star T_R(w)$
is given by
\[
d(a\star 1)=d_Aa\star 1\sands d(1\star w)=z\star 1.
\]
Further cofibrations $A\rightarrow A\mathbin{\star}T_R(W)$ arise
by iteration, for any set $W$ of positive dimensional generators
corresponding to cycles in $A$.

\subsection{Cocommutative differential graded coalgebras}\label{cdgc}
The cocommutative co\-mo\-no\-ids in $\cat{ch}_R$ are the objects
of $\cat{cdgc}_R$, and the morphisms preserve comultiplication.
The model structure is defined only over fields of characteristic
$0$; in view of our applications, we shall restrict attention to
the case $\Q$. In practice, we interpret $\cat{cdgc}$ as the full
subcategory $\cat{cdgc}_0$ of connected objects $C$, which are
necessarily supplemented. Model structure was first defined on the
category $\cat{cdgc}_1$ of simply connected rational cocommutative
coalgebras by Quillen \cite{quil69}, and refined to $\cat{cdgc}_0$
by Neisendorfer \cite{neis78}. The cofibrations are monomorphisms,
and the fibrations are determined by the appropriate lifting
property.

Limits exist because $\cat{cdgc}$ has finite products and
filtered limits, whereas colimits are created in $\cat{ch}$, and
endowed with the natural coalgebra structure. Functoriality of the
factorisations again follows by adapting the proofs of
\cite{bo-gu76}.

For example, the product of coalgebras $C$ and $D$ is their tensor
product $C\otimes D$ over $\Q$. The coproduct is their supplemented
sum, given by pushing out the diagram of supplementations
\[
C\stackrel{\delta_C}{\lllongleftarrow}\Q
\stackrel{\delta_D}{\lllongrightarrow}D
\]
in $\cat{ch}$ and imposing the standard comultiplication on the
result. Examples of fibrations include the projections $(C\otimes
S\br{dw},d)\rightarrow C$, which are determined by cycles $z$ in $C$
and obtained by pulling back diagrams
\[
  C\stackrel{h}{\llongrightarrow} S\br{x}\stackrel{p}\llongleftarrow
  S\br{w,dw},
\]
where $p(w)=x$,\, $p(dw)=0$ and $h(z)=x$. The differential on
$C\otimes S\br{dw}$ satisfies
\[
  d(z\otimes 1)=1\otimes dw \sands d(1\otimes dw)=0.
\]
This illustrates the fact that the pullback of a fibration is a
fibration. Further fibrations $C\otimes S\br{dW}\to C$ are given by
iteration, for any set $W$ of generators corresponding to elements
of degree $\geq 2$ in $C$.

\subsection{Differential graded coalgebras}\label{dgc}

Model structures on more general categories of differential graded
coalgebras have been publicised by Getzler and Goerss \cite{ge-go99},
who also work over a field. Once more, we restrict attention to
$\Q$. The objects of $\cat{dgc}$ are comonoids in $\cat{ch}$, and the
morphisms preserve comultiplication. The cofibrations are
monomorphisms, and the fibrations are determined by the appropriate
lifting property.

Limits exist because $\cat{dgc}$ has finite products and filtered
limits, and colimits are created in $\cat{ch}$. Functoriality of
factorisations follows from the fact that the model structure is
cofibrantly generated.

For example, the product of coalgebras $C$ and $D$ is the cofree
product ${C\star D}$ \cite{ge-go99}. Their coproduct is the
supplemented sum, as in the case of $\cat{cdgc}$. Examples of
fibrations include the projections $(C\star T\br{dw},d)\rightarrow C$,
which are determined by cycles $z$ in $C$ and obtained by pulling back
diagrams
\[
  C\stackrel{h}{\llongrightarrow} T\br{x}\stackrel{p}\llongleftarrow
  T\br{w,dw},
\]
where $p(w)=x$,\, $p(dw)=0$ and $h(z)=x$. The differential on
$C\star T\br{dw}$ satisfies
\[
  d(z\star 1)=1\star dw \sands d(1\star dw)=0.
\]

\subsection{Differential graded Lie algebras}\label{dgl}

For our purposes here, a differential graded Lie algebra $L$ is a
chain complex in $\cat{ch}$, equipped with a bracket morphism
$L\otimes L\rightarrow L$ satisfying signed versions of the
antisymmetry and Jacobi identities. Differential graded Lie algebras
over $\Q$ are the objects of $\cat{dgl}$, and their homomorphisms are
the morphisms. Quillen \cite{quil69} originally defined a model
structure on the subcategory $\cat{dgl}_1$ of {\it reduced\/} objects,
which was extended to $\cat{dgl}$ by Neisendorfer
\cite{neis78}. Fibrations are epimorphisms, and cofibrations are
determined by the appropriate lifting property.

Limits are created in $\cat{ch}$, whereas colimits exist because
$\cat{dgl}$ has finite coproducts and filtered colimits.
Functoriality of the factorisations follows by adapting the proofs
of~\cite{neis78}.

For example, the product of Lie algebras $L$ and $M$ is their product
$L\oplus M$ as chain complexes, with the induced bracket
structure. Their coproduct is the free product $L\star M$, obtained by
factoring out an appropriate differential graded ideal from the free
Lie algebra $\fl(L\otimes M)$ on the chain complex $L\otimes
M$. Examples of cofibrations include the extensions
$L\to(L\star\fl(w),d)$, which are determined by cycles $z$ in $L$ and
obtained by pushing out diagrams
\[
  L\stackrel{h}{\llongleftarrow}\fl(x)\stackrel{j}
  \llongrightarrow\fl(w,dw),
\]
where $h(x)=z$ and $j(x)=dw$. The differential on $L\star F(w)$ is
given by
\[
  d(l\star 1)=d_Ll\star 1\sands d(1\star w)=z\star 1.
\]
For historical reasons, a differential graded Lie algebra $L$ is said
to be \emph{coformal} whenever it is formal in $\cat{dgl}$.

\subsection{Adjoint pairs}\label{adpa}
Following Moore \cite{moor71}, \cite{h-m-s74}, we consider the
\emph{classifying functor} $B_*$ and the \emph{loop functor}
$\varOmega_*$ as an adjoint pair
\begin{equation}\label{barcobar}
  \varOmega_*\colon\cat{dgc}_{0,R}\rla\cat{dga}_R:\! B_*.
\end{equation}
For any object $A$ of $\cat{dga}_R$, the classifying coalgebra $B_*A$
agrees with Eilenberg and Mac Lane's normalised bar construction as
objects of $\cat{ch}_R$. For any object $C$ of $\cat{dgc}_{0,R}$, the
loop algebra $\varOmega_*C$ is given by the tensor algebra
$T_R(s^{-1}\oline{C})$ on the desuspended $R$-module
$\oline{C}=\Ker(\varepsilon\colon C\to R)$, and agrees with Adams's
cobar construction \cite{adam56} as objects of $\cat{ch}_R$. The
graded homology algebra $H(\varOmega_*C)$ is denoted by
$\mathop{\mathrm{Cotor}}_C(R,R)$, and is isomorphic to the Pontrjagin
ring $H_*(\varOmega X;R)$ when $C$ is the singular chain complex of a
reduced CW-complex $X$.

Over $\Q$, the adjunction maps $C\mapsto B_*\varOmega_*C$ and
$\varOmega_*B_*A\mapsto A$ are quasi-\-iso\-mor\-phisms for every
object $A$ and $C$. We explore the consequence of this fact
further in Section \ref{mlsak}. In particular, there is an
isomorphism
\begin{equation}\label{cotorext}
\mathop{\mathrm{Cotor}}\nolimits_C(\Q,\Q)
\;\cong\;\Ext_{C^*}(\Q,\Q)
\end{equation}
of graded algebras \cite[page~41]{prid70}, where $C^*$ is the graded
algebra dual to $C$ and $\Ext_{C^*}(\Q,\Q)$ is the rational
\emph{Yoneda algebra} of $C^*$~\cite{macl63}.

Following Neisendorfer ~\cite[Proposition 7.2]{neis78}, we consider a
second pair of adjoint functors
\begin{equation}\label{lcadj}
L_*\colon\cat{cdgc}_0\rla\cat{dgl}:\! M_*,
\end{equation}
whose derived functors induce an equivalence between
$\Ho(\cat{cdgc}_0)$ and a certan full subcategory of $\Ho(\cat{dgl})$.
This extends Quillen's original results \cite{quil69} for $L_*$ and
$M_*$, which apply only to simply connected coalgebras and connected
Lie algebras. Given a connected cocommutative coalgebra $C$, the
underlying graded Lie algebra of $L_*C$ is the free Lie algebra
$\fl(s^{-1}\oline{C})\subset T(s^{-1}\oline{C})$. This is preserved by
the differential in $\varOmega_*C$ because $C$ is cocommutative,
thereby identifying $L_*C$ as the differential graded Lie algebra of
primitives in $\varOmega_*C$. The right adjoint functor $M_*$ may be
regarded as a generalisation to differential graded objects of the
standard complex for calculating the cohomology of Lie algebras.
Given any $L$ in $\cat{dgl}$ the underlying cocommutative coalgebra of
$M_*L$ is the symmetric coalgebra $C(sL)$ on the suspended vector
space $L$.

%
%
%
%
%
%
%
%
%

\section{Homotopy limits and colimits}\label{holico}

The $\lim$ and $\colim$ functors $\fcat{a}{mc}\to\cat{mc}$ do not
generally preserve weak equivalences, and the theory of homotopy
limits and colimits has been developed to remedy this deficiency. We
outline their construction in this section, and discuss basic
properties. The literature is still in a state of considerable flux,
and we refer to Recke's thesis \cite{reck} for a comparison of several
alternative treatments. Here we focus mainly on those of Recke's
statements that are inspired by Hirschhorn, and make detailed appeal
to \cite{hirs03} as necessary.

With $\cat{cat}(K)$ and $\cat{cat}^{op}(K)$ in mind as primary
examples, we assume throughout that $\cat{a}$ is a finite Reedy
category. One additional property is especially useful.
\begin{definition}\label{deffcfc}
A Reedy category $\cat{a}$ has \emph{cofibrant constants} if the
constant $\cat{a}$-diagram $M$ is Reedy cofibrant in $\fcat{a}{mc}$,
for any cofibrant object $M$ of an arbitrary model category
$\cat{mc}$. Similarly, $\cat{a}$ has \emph{fibrant constants} if the
constant $\cat{a}$-diagram $N$ is Reedy fibrant for any fibrant object
$N$ of $\cat{mc}$.
\end{definition}

Note that the initial and terminal objects of $\fcat{a}{mc}$ are the
constant diagrams $\circ$ and $*$ respectively.

\begin{lemma}\label{catkrcf}
The Reedy categories $\cat{cat}(K)$ and $\cat{cat}^{op}(K)$ have
fibrant and cofibrant constants, for every simplicial complex $K$.
\end{lemma}
\begin{proof}
  Consider $\cat{cat}(K)$, and any face $\s$. Then
  $\sigma\!\Downarrow\!\cat{cat}_-(K)$ is empty, since no morphism
  lowers degree, whereas $\cat{cat}_+(K)\!\Downarrow\!\s$ has initial
  object $\varnothing$, and is therefore connected. So $\cat{cat}(K)$
  satisfies the criteria of \cite[Proposition 15.10.2]{hirs03} for
  fibrant and cofibrant constants. The result for $\cat{cat}^{op}(K)$
  follows by duality.
\end{proof}

As shown in \cite[Theorem 15.10.8]{hirs03}, a Reedy category
$\cat{a}$ has fibrant and cofibrant constants if and only if the
adjoint functors
\begin{equation}\label{quilpairs}
\kappa\colon\cat{mc}\rla\fcat{a}{mc}:\!\lim\spandsp
\colim\colon\fcat{a}{mc}\rla\cat{mc}:\!\kappa
\end{equation}
are Quillen pairs, for every model category $\cat{mc}$. The proof
addresses the equivalent statement that $\kappa$ is both left and
right Quillen. In these circumstances, it follows from Ken Brown's
Lemma that $\lim$ and $\colim$ preserve weak equivalences on Reedy
fibrant and Reedy cofibrant diagrams respectively.

We now apply the fibrant and cofibrant replacement functors associated
to the Reedy model structure on $\fcat{a}{mc}$, and their homotopy
functors \eqref{hoomfi}.
\begin{definition}\label{defhocolhol}
For any Reedy category $\cat{a}$ with fibrant and cofibrant
constants, and any model category $\cat{mc}$:
\begin{enumerate}
\item[(1)] the \emph{homotopy colimit} functor is the composition
\[
\hocolim\colon\Ho\fcat{a}{mc}
\stackrel{\Ho(\omega)}{\llllongrightarrow}
\Ho\fcat{a}{mc}_c
\stackrel{\Ho(\colim)}{\lllllongrightarrow}\Ho(\cat{mc})\;;
\]
\item[(2)] the \emph{homotopy limit} functor is the composition
\[
\holim\colon\Ho\fcat{a}{mc}
\stackrel{\Ho(\phi)}{\llllongrightarrow}
\Ho\fcat{a}{mc}_{\!f}
\stackrel{\Ho(\lim)}{\lllllongrightarrow}\Ho(\cat{mc}).
\]
\end{enumerate}
\end{definition}

Lemma \ref{catkrcf} confirms that $\holim$ and
$\hocolim\colon\Ho\fcat{cat$(K)$}{mc}\rightarrow\Ho(\cat{mc})$ are
defined for any simplicial complex $K$; and similarly for
$\cat{cat}^{op}(K)$.

\begin{remark}\label{htpyinvce}
Definition \ref{defhocolhol} incorporates the fact that $\holim$ and
$\hocolim$ map objectwise weak equivalences of diagrams to weak
equivalences in $\cat{mc}$.
\end{remark}

Describing explicit models for homotopy limits and colimits has been a
major objective for homotopy theorists since their study was initiated
by Bousfield and Kan \cite{bo-ka72}. In terms of Definition
\ref{defhocolhol}, the issue is to choose fibrant and cofibrant
replacement functors $\phi$ and $\omega$. Many alternatives exist,
including those defined by the two-sided bar and cobar constructions
of \cite{p-r-v04} or the frames of \cite[\S16.6]{hirs03}, but no
single description yet appears to be convenient in all cases. Instead,
we accept a variety of possibilities, which are often implicit; the
next few results ensure that they are as compatible and well-behaved
as we need.
\begin{proposition}\label{hclhlpreswe}
Any Reedy cofibrant approximation $D'\to D$ of diagrams induces a weak
equivalence $\colim D'\to\hocolim D$ in $\cat{mc}$; and any Reedy
fibrant approximation $D\to D''$ induces a weak equivalence $\holim
D\to\lim D''$.
\end{proposition}
\begin{proof}
  A cofibrant approximation $i\colon D'\to D$ factorises as
  $D'\to\omega(D)\to D$ by the lifting axiom in $\fcat{a}{mc}$, and
  the left hand map is a weak equivalence by the $2$ out of $3$ axiom.
  But $D'$ and $\omega(D)$ are cofibrant, and $\colim$ is left
  Quillen, so the induced map $\colim D'\to\colim\omega(D)$ is a weak
  equivalence, as required. The proof for $\lim$ is dual.
\end{proof}
\begin{remark}\label{unirepl}
Such arguments may be strengthened to include uniqueness statements,
and show that the replacements $\phi(D)$ and $\omega(D)$ are
themselves unique up to homotopy equivalence over $D$
\cite[Proposition 8.1.8]{hirs03}.
\end{remark}
\begin{proposition}\label{hotolim}
For any Reedy cofibrant diagram $D$ and fibrant diagram $E$, there are
natural weak quivalences $\hocolim D\to\colim D$ and $\lim E\to\holim
E$.
\end{proposition}
\begin{proof}
For $D$, it suffices to apply the left Quillen functor $\colim$ to the
acyclic fibration $\omega(D)\to D$. The proof for $E$ is dual.
\end{proof}
\begin{proposition}\label{hopu}
  In any model category $\cat{mc}$:
\begin{enumerate}
\item[(1)] if all three objects of a pushout diagram $D\letbe
  L\leftarrow M\to N$ are cofibrant, and either of the maps is a
  cofibration, then there exists a weak equivalence $\hocolim
  D\to\colim D$;
\item[(2)] if all three objects of a pullback diagram $E\letbe P\to
  Q\leftarrow R$ are fibrant, and either of the maps is a fibration,
  then there exists a weak equivalence $\holim E\to\lim E$.
\end{enumerate}
\end{proposition}
\begin{proof} For (1), assume that $M\to N$ is a cofibration, and that
  the indexing category $\cat{b}$ for $D$ has non-identity morphisms
  $\lambda\leftarrow\mu\to\nu$. The degree function $\deg(\lambda)=0$,
  $\deg(\mu)=1$, and $\deg(\nu)=2$ turns $\cat{b}$ into a Reedy
  category with fibrant constants, and ensures that $D$ is Reedy
  cofibrant. So Proposition \ref{hotolim} applies. If $M\to L$ is a
  cofibration, the corresponding argument holds by symmetry.

  For (2), the proofs are dual.
\end{proof}
Further details  may be found in \cite[Proposition 19.9.4]{hirs03}.

An instructive example of cofibrant replacement is provided by the
model category $\cat{top}$, whose every object $X$ admits a \emph{CW
model}. It consists of a cofibrant approximation $f\colon W_X\to X$,
where $W_X$ is filtered by skeleta $W^{(n)}$; these are constructed
inductively by the adjunction of euclidean cells, using repeated
pushouts of the form $W^{(n)}\leftarrow S^n\rightarrow D^{n+1}$. Two
weakly equivalent topological spaces have homotopy equivalent
CW-models. Such models have several advantages; for example, the
cellular chain complex $C\hspace{-1pt}e_*(W_X)$ is far more economical
than the singular version $C_*(X)$ for computing the homology of $X$.
They may, however, be difficult to make explicit, and are not
functorial. A genuine cofibrant replacement functor $\omega(X)\to X$
must be constructed with care, and is defined in \cite[\S
98]{dw-sp95}, for example.

An analogous example is given by $\cat{cdga}$, whose every
homologically connected object $A$ admits a \emph{minimal model}
\cite{sull78}. It consists of a cofibrant approximation $f\colon
M_A\to A$, where $M_A$ is connected, is free as a commutative graded
algebra, and contains a natural sequence of subalgebras that are
constructed by successive pushouts of the form \eqref{frexcdga}. It
follows from these requirements that the differential of $M_A$ takes
decomposable values only. Any two minimal models for $A$ are
necessarily isomorphic, and $M_A$ and $M_B$ are isomorphic for
quasi-isomorphic $A$ and $B$. So every zig-zag $A\to\dots\leftarrow B$
in $\cat{cdga}$ may be replaced by a diagram $A\leftarrow M\to B$,
where $M$ is minimal for both $A$ and $B$. The advantage of $M_A$ is
that it simplifies many calculations concerning $A$; disadvantages
include the fact that it may be difficult to describe for relatively
straightforward objects $A$, and that it cannot be chosen
functorially. A genuine cofibrant replacement functor requires
additional care, and seems first to have been made explicit in
\cite[\S 4.7]{bo-gu76}.

%
%
%
%
%
%
%
%
%

\section{Models for $\djs(K)$}\label{modjs}

In this section we introduce algebraic and geometrical models for the
\daja\ spaces $\djs(K)$. Several of the results appear in
\cite{bu-pa02}, \cite{no-ra05}, and \cite{p-r-v04}, but our current
aim is to display them in the model categorical setting; they are then
more readily comparable with calculations in later sections. Analogous
constructions may be made in the real case by substituting $\Z/2$ for
$T$ \cite{da-ja91}.

Over $R$, the graded \emph{Stanley-Reisner algebra} \cite{stan96} of
an arbitrary simplicial complex $K$ is given by the quotient
\begin{equation}\label{srdef}
R[K]\;=\;S_R(V)/(v_\zeta\colon\zeta\notin K),
\end{equation}
otherwise known as the \emph{face ring} of $K$. Any quotient of
$S_R(V)$ by a square-free monomial ideal is the Stanley-Reisner
algebra of some simplicial complex.

The \emph{Stanley-Reisner coalgebra} $R\br{K}$ is the graded dual of
$R[K]$, and is less well-known. To aid its description, we
write $K_\bullet$ for the simplicial set generated by $K$
\cite[Example 1.4]{may67}, whose non-degenerate simplices are the
faces of $K$. Then $R\br{K}$ is the free $R$-module on generators
$v\br{\s}$, as $\s$ varies over the simplices of $K_\bullet$.  The
coproduct takes the form
\begin{equation}\label{srdiag}
  \delta v\br{\s}=
  \sum_{\s=\tau\,\sqcup\,\tau'}v\br{\tau}\otimes v\br{\tau'},
\end{equation}
where the sum ranges over all partitions of $\s$ into subsimplices
$\tau$ and $\tau'$. By construction, the elements $v\br{\s}$ form the
basis dual to the generating monomials $v_\s$ of $R[K]$.

The \emph{exterior face ring} $R_\wedge[K]$ is defined by analogy
with \eqref{srdef}, on the $1$--di\-men\-sio\-nal vertices $U$. So
long as $\frac{1}{2}\in R$, its dual $R^\wedge\br{K}$ is the free
$R$-module on generators $u\br{\s}$, as $\s$ varies over the faces
of $K$; the diagonal is also given by \eqref{srdiag}.

In the context of algebraic toric objects, Stanley-Reisner algebras
and coalgebras may also be described as limits and colimits of
$\cat{cat}(K)$ diagrams \cite{p-r-v04}. We invest them with zero
differential, and work in the categories $\cat{dga}_R$ and
$\cat{dgc}_R$ respectively. In $\cat{dga}_R$ we define the
$\cat{cat}^{op}(K)$-diagram $S_K(V)$, whose value on $\tau\supseteq\s$
is the projection $S_R(\tau)\to S_R(\s)$; and in $\cat{dgc}_R$ we
define the $\cat{cat}(K)$-diagram $S^K\br{V}$, whose value on
$\s\subseteq\tau$ is the inclusion $S_R\br{\sigma}\to S_R\br{\tau}$ of
coalgebras. Then there are isomorphisms
\begin{equation}\label{srlim}
  R[K]\stackrel{\cong}\longrightarrow\lim\nolimits^{\scat{cdga}_R}S_K(V)
  \sands
  \colim^{\scat{cdgc}_R}S^K\br{V}\stackrel{\cong}{\longrightarrow}R\br{K}
\end{equation}
(here and below the superscript in the notation for limit and
colimit refers to the target category). The limits and colimits of
\eqref{srlim} are created in $\cat{ch}_R$ and the additional
algebraic structure is superimposed; so they may equally well be
taken over $\cat{dga}_R$ and $\cat{dgc}_R$.

In the rational case, $\cat{cdga}$ and $\cat{cdgc}$ admit the model
structures of subsections \ref{cdga} and \ref{cdgc}, so the
corresponding homotopy limits and homotopy colimits exist. Since the
diagrams $S_K(V)$ and $S^K\br{V}$ are Reedy fibrant and Reedy
cofibrant respectively, we may rewrite \eqref{srlim} in terms of weak
equivalences
\begin{equation}\label{srhlim}
  \Q[K]\stackrel{\simeq}{\longrightarrow}\holim\nolimits^{\scat{cdga}}S_K(V)
  \sands
  \hocolim^{\scat{cdgc}}S^K\br{V}\stackrel{\simeq}{\longrightarrow}\Q\br{K},
\end{equation}
by applying Proposition \ref{hotolim}. Similar remarks apply to the
exterior cases.

In \cite{da-ja91}, \daaja\ introduce the space now known as $\djs(K)$
by means of a Borel construction. It is designed to ensure that the
cohomology ring $H^*(\djs(K);\Z)$ is isomorphic to $\Z[K]$, and
therefore that the homology coalgebra $H_*(\djs(K);\Z)$ is isomorphic
to $\Z\br{K}$ by duality.

In~\cite{bu-pa02} and \cite{p-r-v04}, the $\cat{cat}(K)$ diagram
$BT^K$ is defined in $\cat{top}$; its colimit
\begin{equation}\label{cxdjco}
  \colim BT^K\;=\;\bigcup_{\sigma\in K}BT^\sigma
\end{equation}
is a subcomplex of $BT^V$, and is shown to be homotopy equivalent to
$\djs(K)$. Here $T^\s<T^V$ is the coordinate subtorus, so its
classifying space $BT^{\s}$ is automatically a subcomplex of
$BT^V\simeq(C P^\infty)^V$. The colimit inherits a natural cell
structure from $BT^V$, whose cellular chain coalgebra may readily be
identified with $R\br{K}$ for any coefficient ring $R$. Since $BT^K$
is Reedy cofibrant, there is a weak equivalence $\hocolim
BT^K\to\djs(K)$, by Proposition \ref{hotolim}.

Similar remarks apply to the exterior case, by replacing the circle
$T$ with $\Z$. So the diagram $B\Z^K$ may be identified with $T^K$,
and is also Reedy cofibrant; its colimit $\djs_\wedge(K)$ is a
subcomplex of the torus $T^V$, and is necessarily finite. By
construction, $H^*(\djs_\wedge(K);R)$ and $H_*(\djs_\wedge(K);R)$ are
isomorphic to $R_\wedge[K]$ and $R^\wedge\br{K}$ respectively.

Sullivan's functor $\APL$ usually provides more sophisticated
contravariant models, in terms of rational cochain algebras. However,
the space $\djs(K)$ is formal \cite{no-ra05} for every complex $K$,
and we may use its minimal model to reduce the zig-zag of
quasi-isomorphisms \eqref{zigzag} to the form
\begin{equation}\label{djkminmod}
  \Q[K]\stackrel{e_K\!}{\llongleftarrow}M(\djs(K))
  \stackrel{\simeq}{\llongrightarrow}\APL(\djs(K)
\end{equation}
in $\cat{cdga}$. In this sense, $\Q[K]$ is just as good a model for
$\djs(K)$ as $\APL(\djs(K))$; for a discussion of uniqueness, we refer
to \cite[Proposition 5.10]{no-ra05}. Results on the coformality of
$\djs(K)$ are presented in Section~\ref{mlsfk}.

We summarise this section in terms of our motivating principle, that
certain functors preserve homotopy limit and colimit structures on
specific toric objects.
\begin{proposition}\label{firstinvt}
  As functors $\cat{top}\to\cat{cdgc}$ and $\cat{top}\to\cat{cdga}$
  respectively, rational homology preserves homotopy colimits and
  rational cohomology maps homotopy colimits to homotopy limits,
  on diagrams $T^K$ and $BT^K$.
\end{proposition}
Proposition \ref{firstinvt} generalises immediately to the integral
situation, although the model categorical framework has to be relaxed
in the coalgebraic case.

%
%
%
%
%
%
%
%
%

\section{Models for $\zk$}\label{mozk}

In this section we introduce algebraic and geometrical models for the
moment angle complexes $\zk$ of \cite{da-ja91}, which form a second
important class of toric spaces associated to a simplicial complex
$K$.

Many applications to toric manifolds and combinatorial commutative
algebra are developed in \cite{bu-pa02}, where $\zk$ is described as
the homotopy fibre of the inclusion $\djs(K)\to BT^V$. As such, it is
obtained by pulling back the diagram
\begin{equation}\label{zkpul}
  P_K\;\letbe\;\djs(K)
  \stackrel{\scriptscriptstyle\subseteq}{\llongrightarrow}BT^V
  \stackrel{p}{\llongleftarrow}ET^V
\end{equation}
in $\cat{top}$, and therefore by restricting the universal
$T^V$-bundle $ET^V\rightarrow BT^V$ to $\djs(K)$. So $\zk$ inherits a
canonical $T^V$-action, and there exists a weak equivalence
$\zk\to\holim P_K$ by Proposition \ref{hopu}(2).

Alternatively, $\zk$ may be identified with the subcomplex
$\bigcup_{\sigma\in K}D_\s$ of the standard unit $2m$-disk
$(D^2)^V\subset\C^V$, where
\begin{equation}\label{zkdecomp}
  D_\s\;=\;\bigl\{(z_1,\ldots,z_m)\in(D^2)^V\colon |z_j|=1\text{ for }
  j\notin\s \bigr\}.
\end{equation}
Note that $D_\s\cong(D^2)^s\times T^{m-s}$ when $|\s|=s$, for
every $0\leq s\leq m$. A convenient cellular structure on $\zk$ is
given by combining this description with the decomposition of
$D^2$ into a single cell in each dimension $0$, $1$, and $2$. If
$K$ is a triangulated sphere then $\zk$ is homotopy equivalent to
a manifold~\cite[Lemma 6.13]{bu-pa02}.

We consider $\cat{cat}(K)$-diagrams $D^K$ and $T^{V/K}$, whose values
on $\s\subseteq\tau$ are the inclusion $D_\s\subseteq D_\tau$ and the
quotient $T^V/T^\s\to T^V/T^\tau$ respectively. In particular, $\colim
D^K$ is $\zk$. Objectwise projection induces a weak equivalence
$D^K\rightarrow T^{V/K}$ in $\fcat{cat$(K)$}{top}$, whose source is
Reedy cofibrant but whose target is not. Proposition \ref{hclhlpreswe}
therefore determines a weak equivalence
\begin{equation}\label{zkhocolim}
  \colim D^K\stackrel{\simeq}\longrightarrow\hocolim T^{V/K},
\end{equation}
which agrees with the equivalence $\zk\simeq\hocolim T^{V/K}$ of
\cite[Corollary 5.4]{p-r-v04}.

So we have described $\zk$ as $\holim P_K$ and as $\hocolim
T^{V/K}$. A weak equivalence between the two is given in the
categorical context by \cite[Proposition~5.1]{p-r-v04}, and we shall
apply $\APL$ to each. We employ the standard model \cite[\S
15(c)]{f-h-t01} for the principal $T^V$-bundle of \eqref{zkpul}, given
by the commutative diagram
\begin{equation}\label{btvmod}
\begin{CD}
S(V)@>>>(\wedge(U)\otimes S(V),d)\\ @V\simeq VV@VV\simeq V\\
\APL(BT^V)@>p^*>>\APL(ET^V)
\end{CD}
\end{equation}
in $\cat{cdga}$. The differential satisfies $du_j=v_j$ for $1\leq
j\leq m$; also, $p^*$ is a cofibration of the form described in
Subsection \ref{cdga}, and $S(V)$ is the minimal model for $BT^V$.

First we apply $\APL$ to~\eqref{zkpul}, on the understanding that it
does not generally convert pullbacks to pushouts \cite[\S3]{bo-gu76}.
We obtain the $\cat{cdga}$-diagram
\[
  P'_K\;\letbe\;\APL(\djs(K))\longleftarrow \APL(BT^V)\longrightarrow
  \APL(ET^V),
\]
which is not Reedy cofibrant.
\begin{theorem}\label{zkmodel2}
  The algebras $\APL(\zk)$ and $\hocolim P'_K$ are weakly equivalent
  in $\cat{cdga}$.
\end{theorem}
\begin{proof}
  By \eqref{btvmod}, there is an objectwise weak equivalence mapping
  the diagram
\begin{equation}\label{ppk}
 P''_K\;\letbe\; M(\djs(K))\stackrel{\;q}{\longleftarrow}S(V)
\longrightarrow\wedge(U)\otimes S(V)
\end{equation}
to $P'_K$. Moreover, $\colim P''_K$ is given by $(\wedge(U)\otimes
M(\djs(K)),d)$, where $du_j=q(v_j)$ for $1\leq j\leq m$, and $\colim
P'_K$ maps naturally to $\APL(\zk)$; so the composition
$e\colon\wedge(U)\otimes M(\djs(K))\to\APL(\zk)$ is defined.
Since $\djs(K)$ is simply connected, $e$ is a quasi-isomorphism by
\cite[\S15(c)]{f-h-t01}.

The right hand arrow of $P''_K$ is a cofibration and its objects are
cofibrant, so there is a quasi-isomorphism $\hocolim P''_K\to\colim
P''_K$ by Proposition \ref{hopu}(1). Furthermore, the induced map
$\hocolim P''_K\to\hocolim P'_K$ is a quasi-isomorph\-ism by Remark
\ref{htpyinvce}. So there is a zig-zag
\[
\hocolim P'_K\longleftarrow\hocolim P''_K\longrightarrow\APL(\zk)
\]
of quasi-isomorphisms in $\cat{cdga}$, as required.
\end{proof}

In order to apply $\APL$ to \eqref{zkhocolim}, we consider
$\cat{cat}^{op}(K)$-diagrams $\wedge(U)\otimes S_K(V)$ and
$\wedge_{U/K}$ in $\cat{cdga}$. The values of the former on
$\tau\supseteq\s$ are the quotient maps
\[
  (\wedge(U)\otimes S(\tau),d)\longrightarrow
  (\wedge(U)\otimes S(\s),d),
\]
where $d$ is defined on $\wedge(U)\otimes S(\s)$ by $du_j=v_j$ for
$u_j\in\s$, and $0$ otherwise; the values of the latter are the
monomorphisms $\wedge(u_i\notin\tau) \rightarrow
\wedge(u_i\notin\s)$ in $\cat{cdga}$. Objectwise projection
induces a weak equivalence
\begin{equation}\label{uskuk}
\wedge(U)\otimes S_K(V)\stackrel{\simeq}{\longrightarrow}
\wedge_{U/K}
\end{equation}
in $\fcat{cat$^{op}(K)$}{cdga}$, whose source is Reedy fibrant but
whose target is not.
\begin{theorem}\label{zkmodel1}
  The algebras $\APL(\zk)$ and $\holim\wedge_{U/K}$ are weakly
  equivalent in $\cat{cdga}$.
\end{theorem}
\begin{proof}
Choosing representative cocycles for generators of $H^1(T^V;\Q)$
yields compatible quasi-isomorphisms
$\wedge(u_i\notin\s)\to\APL(D_\s)$, as $\s$ ranges over the faces of
$K$. These define a weak equivalence $\wedge_{U/K}\to\APL\circ D^K$ of
$\cat{cat}^{op}(K)$-diagrams, which combines with \eqref{uskuk} to
give a weak equivalence
\[
 \wedge(U)\otimes S_K(V)\stackrel{\simeq}{\longrightarrow}\APL\circ
 D^K
\]
of Reedy fibrant diagrams in $\cat{cdga}$. Their limits are therefore
quasi-isomorphic, since $\lim$ is right Quillen. On the other hand,
$A^*$ maps colimits to limits \cite[\S13.5]{bo-gu76}), so there exists
a zig-zag of quasi-isomorphisms
\begin{equation}\label{alim}
  \lim(\APL\circ D^K)\lzze\APL(\zk)
\end{equation}
in $\cat{cdga}$, by analogy with \cite[(5.6)]{no-ra05}. Hence
$\lim(\wedge(U)\otimes S_K(V))$ and $\APL(\zk)$ are weakly equivalent.

Finally, \eqref{uskuk} and Proposition \ref{hotolim} provide a
quasi-isomorphism
\begin{equation}\label{holimuk}
\lim(\wedge(U)\otimes S_K(V))
\stackrel{\simeq}{\longrightarrow}
\holim\wedge_{U/K}
\end{equation}
in $\cat{cdga}$.
\end{proof}

Our two algebraic models for $\zk$ are consistent, in the
following sense. Define $(\wedge(U)\otimes\Q[K],\,d)$ in $\cat{cdga}$
by $du_j=v_j$ for $1\leq j\leq m$, and recall the quasi-isomorphism
$e_K$ of \eqref{djkminmod}; then \cite[Lemma 14.2]{f-h-t01} shows
that
\begin{equation}
1\otimes e_K\colon\wedge(U)\otimes M(\djs(K))\llongrightarrow
\wedge(U)\otimes\Q[K]
\end{equation}
is a quasi-isomorphism. So $\wedge(U)\otimes\Q[K]$ is
quasi-isomor\-phic to $\hocolim P'_K$. On the other hand, it is
isomorphic to $\lim(\wedge(U)\otimes S_K(V))$ because limits are
created in $\cat{coch}$. So $\wedge(U)\otimes\Q[K]$ is also
quasi-isomorphic to $\holim\wedge_{U/K}$.

We may now retrieve the calculations of~\cite{bu-pa02} for
$H^*(\zk;\Q)$.
\begin{corollary}\label{ratzk}
There are isomorphisms of (bi)graded algebras
\[
  H^*(\mathcal Z_K;\Q)\;\cong\;\Tor_{S(V)}(\Q[K],\Q)
  \;\cong\; H(\wedge(U)\otimes\Q[K],d).
\]
\end{corollary}

\begin{remarks}
The algebra $(\wedge(U)\otimes S(\s),d)$ is actually the cellular
cochain algebra of $(S^\infty)^\s\times T^{V\setminus\s}$, where
$S^\infty$ has a single cell in each dimension. So the projection
$\wedge(U)\otimes S(\s)\to \wedge(u_i\notin\s)$ is an algebraic model
for the retraction $(S^\infty)^\s\times T^{V\setminus\s}\to
T^{V\setminus\s}$, given any face $\s$ of $K$. The corresponding model
for the inclusion $D_\s\to(S^\infty)^\s\times T^{V\setminus\s}$ is the
projection
\[
  \wedge(U)\otimes S(\s)\to
  \wedge(U)\otimes S(\s)\bigr/
  (u_jv_j=v_j^2=0\text{ for $j\in\s$}),
\]
which reflects the fact that $(u_jv_j=v_j^2=0\text{ for $j\in\s$})$ is
an acyclic ideal~\cite{b-b-p04}. These models also work over $\Z$, and
are used in~\cite{b-b-p04} to establish an integral version of
Corollary~\ref{ratzk} (which was confirmed in ~\cite{fran06} by other
methods). The model categorical interpretation must be relaxed in
this case; nevertheless, $\wedge_R(U)\otimes R[K]$ may still be
interpreted as a homotopy limit in $\cat{dga}_R$.

As shown in \cite{bask03}, the algebra $\wedge(U)\otimes\Q[K]$ admits
non-trivial triple Massey products for certain complexes $K$, in which
case $\zk$ cannot be formal.
\end{remarks}

We may summarise this section by combining Theorems \ref{zkmodel2}
and \ref{zkmodel1}.
\begin{proposition}\label{secondinvt}
  As functors $\cat{top}\to\cat{cdga}$, both rational cohomology and
  $\APL$ map homotopy limits to homotopy colimits on diagrams $P_K$,
  and map homotopy colimits to homotopy limits on diagrams $T^{V/K}$.
\end{proposition}

%
%
%
%
%
%
%
%
%

\section{Models for quasitoric manifolds}\label{moqtm}

In this section we describe the properties of \daaja's quasitoric
manifolds within the model categorical framework. We prove that they
are rationally formal, and extend our analysis to generalisations such
as the torus manifolds of \cite{ma-pa06}.

For any simplicial convex $n$--polytope, the Stanley-Reisner algebra
of the boundary complex $K$ is \emph{Cohen--Macaulay} \cite{br-he98},
and $\Q[K]$ admits a $2$-dimensional \emph{linear system of
parameters} $l_1$, \dots, $l_n$. If the parameters are integral, then
an associated quasitoric manifold $M[l]$ may be constructed; up to
homotopy, it is the pullback of the diagram
\begin{equation}\label{qtn}
P_l\;\letbe\;\djs(K)\stackrel{l}{\llongrightarrow}BT^n
\stackrel{p}{\llongleftarrow}ET^n,
\end{equation}
where $l$ represents the sequence $(l_1,\dots,l_n)$ as an element
of $H^2(\djs(K);\Z^n)$. In these circumstances, the quotient of
$M[l]$ by the canonical $T^n$-action is the simple polytope whose
boundary is dual to $K$, and $\Z[K]$ is free and of finite rank
over the polynomial algebra $S_\Z(L)$. Every object of $\cat{top}$
is fibrant and $p$ is a fibration, so there is a weak equivalence
$M[l]\to\holim P_l$ by Proposition \ref{hopu}(2).

In $2$-dimensional integral homology, $l$ induces
$l_*\colon\Z^V\rightarrow \Z^n$, which extends to a
\emph{dicharacteristic} homomorphism $\ell\colon T^V\rightarrow T^n$
\cite{bu-ra01}. The kernel of $\ell$ is an $(m-n)$-dimensional
subtorus $T[l]<T^V$, which acts freely on $\zk$ with quotient $M[l]$.
Moreover, $\ell$ maps $T^\sigma$ isomorphically onto its image, which
we denote by $T(\s)<T^n$ for any face $\s$ of $K$. We write the
integral cohomology ring $H^*(T^n/T(\s))$ as $\wedge(\s,l)$, because
it is isomorphic to an exterior subalgebra on $n-|\s|$ generators.

Following \cite[Proposition 5.3]{w-z-z99}, we may also describe $M[l]$
as the homotopy colimit of the $\cat{cat}(K)$-diagram $T^{n/(K,l)}$,
whose value on $\s\subseteq\tau$ is the projection $T^n/T(\s)\to
T^n/T(\tau)$. This result appears to be the earliest mention of
homotopy colimits in the toric context, and refers to a diagram that
is clearly not Reedy cofibrant. The weak equivalence of $\hocolim
T^{n/(K,l)}$ with the pullback of \eqref{qtn} follows directly from
\cite[Prop.~5.1]{p-r-v04}, and is the quotient of the equivalence
between \eqref{zkpul} and \eqref{zkhocolim} by $T[l]$, which acts
freely on the entire construction.

As in Section \ref{modjs}, we apply $\APL$ to both descriptions. We
need the $\cat{cat}^{op}(K)$-diagram $\wedge_{n/(K,l)}$, which maps
$\tau\supseteq\s$ to the inclusion $\wedge(\tau,l)\to\wedge(\s,l)$
in $\cat{cdga}$, and is clearly not Reedy fibrant. We also use the
standard model \cite[Chapter 15(c)]{f-h-t01} for the principal
$T^n$-bundle of \eqref{qtn}, given by the commutative diagram
\begin{equation}\label{btnmod}
\begin{CD}
S(Y)@>>>(\wedge(X)\otimes S(Y),d)\\ @V\simeq VV@VV\simeq V\\
\APL(BT^n)@>p^*>>\APL(ET^n)
\end{CD}
\end{equation}
in $\cat{cdga}$;
here $X$ and $Y$ denote sets of $1$-- and $2$--dimensional variables
$x_1,\ldots,x_n$ and $y_1,\dots,y_n$ respectively, with $dx_i=y_i$ for
$1\leq i\leq n$. The upper arrow is a cofibration, as described in
Subsection \ref{cdga}, and $S(Y)$ is the minimal model for $BT^n$, as
in \eqref{btvmod}.

We proceed by generalising arguments of Bousfield and Gugenheim
\cite[\S16]{bo-gu76}. We work with an abitrary simply connected
CW-complex $B$ of finite type, and a set of cohomology classes $l_1$,
\dots, $l_n$ in $H^2(B;\Z)$. By analogy with \eqref{qtn}, we let
$l\colon B\to BT^n$ represent $(l_1,\dots, l_n)$, and define
$E[l]$ as the colimit of
\begin{equation}\label{etn}
  B\stackrel{l}{\llongrightarrow}BT^n
  \stackrel{p}{\llongleftarrow}ET^n.
\end{equation}
Thus $l^*(y_i)=l_i$ in $H^2(B;\Z)$ for $1\leq i\leq n$, and
$H^*(B;\Z)$ is an $S_\Z(Y)$-module.

\begin{proposition}\label{fromfib}
  If $B$ is a formal space and $H^*(B)$ is free over $S(Y)$, then
  $E[l]$ is also formal.
\end{proposition}
\begin{proof}
The minimal model $M(B)$ reduces the zig-zag \eqref{zigzag} to the form
\[
  H^*(B)\stackrel{e}{\longleftarrow}M(B)\longrightarrow \APL(B),
\]
and contains $2$--dimensional cocycles $m_i$ such that $e(m_i)=l_i$
for $1\leq i\leq n$. Applying $\APL$ to \eqref{etn} and importing
\eqref{btnmod} yields the commutative ladder
\begin{equation}\label{bladder}
\begin{CD}
  M(B)@<<<S(Y)@>>>\wedge(X)\otimes S(Y)\\
  @V\simeq VV @VV\simeq V @VV\simeq V\\
  \APL(B)@<l^*<<\APL(BT^n) @>p^*>>\APL(ET^n)
\end{CD}
\end{equation}
in $\cat{cdga}$. The colimit of the upper row is given by
$(\wedge(X)\otimes M(B),d)$, where $dx_i=m_i$ for $1\leq i\leq n$,
and the colimit of the lower row maps naturally to $\APL(E[l])$; so
the composition $e'\colon\wedge(X)\otimes M(B)\to\APL(E[l])$ is
defined.  Since $B$ is simply connected, $e'$ is a quasi-isomorphism
by \cite[\S15(c)]{f-h-t01}, and we obtain a zig-zag
\begin{equation}\label{zzey}
(\wedge(X)\otimes H^*(B),d)\stackrel{1\otimes e}{\llongleftarrow}
\wedge(X)\otimes M(B)\stackrel{e'}{\llongrightarrow}\APL(E[l])
\end{equation}
in $\cat{cdga}$, where $dx_i=l_i$ in $H^2(B)$ for $1\leq i\leq n$.
Furthermore, $1\otimes e$ is a quasi-isomorphism by \cite[Lemma
14.2]{f-h-t01}.

We now utilise the fact that $H^*(B)$ is free over $S(Y)$.  Taking
the quasi-\-iso\-mor\-phism $\wedge(X)\otimes S(Y)\to\Q$ and
applying the functor $\otimes_{S(Y)}H^*(B)$ yields a
quasi-isomorphism $(\wedge(X)\otimes H^*(B),d)\to H^*(B)/(Y)$,
which is given by projection onto the second factor.  Moreover,
the Eilenberg--Moore spectral sequence confirms that the natural
map $H^*(B)/(Y)\to H^*(E[l])$ is an isomorphism (as explained in
\cite[Lemma~2.1]{ma-pa06}, for example). Combining the resulting
quasi-isomorphism $\wedge(X)\otimes H^*(B)\to H^*(E[l])$ with
\eqref{zzey} then provides a zig-zag
\[
H^*(E[l])\lzze\APL(E[l])
\]
in $\cat{dgca}$, as required.
\end{proof}

\begin{corollary}\label{toricformal}
\sloppy Every quasitoric manifold $M[l]$ is formal, and the
algebra $\APL(M[l])$ is weakly equivalent to
$\holim\wedge_{n/(K,l)}$ in $\cat{cdga}$.
\end{corollary}

\begin{proof}
The formality of $\APL(M[l])$ follows by applying Proposition
\ref{fromfib} to \eqref{qtn}, and \eqref{zzey} confirms that it is
weakly equivalent to $(\wedge(X)\otimes \Q[K],d)$ in $\cat{cdga}$,
where $dx_i=l_i$ for $1\leq i\leq n$. Moreover, the natural map
\[
\wedge(X)\otimes\Q[K]\longrightarrow\lim(\wedge(X)\otimes S_K(Y))
\]
is an isomorphism, using \eqref{srlim} and the fact that limits are
created in $\cat{coch}$.

Objectwise projection induces a weak equivalence $\wedge(X)\otimes
S_K(Y)\stackrel{\simeq}{\longrightarrow}\wedge_{n/(K,l)}$ in
$\fcat{cat$^{op}(K)$}{cdga}$, whose source is Reedy fibrant; so
Proposition \ref{hotolim} provides a quasi-isomorphism
$\lim(\wedge(X)\otimes S_K(Y))\to\holim\wedge_{n/(K,l)}$. The necessary
zig-zag
\[
\holim\wedge_{n/(K,l)}\lzze\APL(M[l])
\]
in $\cat{cdga}$ is then complete.
\end{proof}

So we have proven the following analogue of Proposition \ref{firstinvt}.
\begin{proposition}\label{thirdinvt}
  As a functor $\cat{top}\to\cat{cdga}$, rational cohomology maps
  homotopy colimits to homotopy limits on diagrams $T^{n/(K,l)}$.
\end{proposition}
Proposition \ref{thirdinvt} also extends to the integral setting; in
particular, $H^*(M[l];R)$ is isomorphic to $R[K]/(L)$ for all
coefficient rings $R$.

Similar arguments apply more generally to \emph{torus manifolds} over
\emph{homology polytopes} ~\cite{ma-pa06}, and even to arbitrary torus
manifolds with zero odd dimensional cohomology. In the latter case,
$\Z[K]$ is replaced by the face ring $\Z[\mathcal S]$ of an
appropriate \emph{simplicial poset} $\mathcal S$. This also admits a
linear system of parameters $l$ and is free over $\Z[L]$, so
Proposition~\ref{fromfib} again establishes formality. Non-singular
compact toric varieties are included in this framework, although the
formality of projective examples follows immediately from the fact
that they are K\"ahler.

%
%
%
%
%
%
%
%
%

\section{Models for loop spaces: arbitrary $K$}\label{mlsak}

{\sloppy We now turn our attention to algebraic models for the
Moore loop spaces $\varOmega\djs(K)$, $\varOmega\zk$ and
$\varOmega M$, where $K$ is an arbitrary simplicial complex and
$M$ a quasitoric manifold. Since composition of Moore loops is
strictly associative, these spaces are topological monoids with
identity, and are of independent interest to homotopy theorists.
Their properties are considerably simplified when $K$ satisfies
the requirements of a flag complex, but we postpone discussion of
this situation until the following section.

}
Geometric models for $\varOmega M$ \cite{cartphd07} are
currently under development.

Following \cite{p-r-v04}, we loop the fibrations of Sections
\ref{mozk} and \ref{moqtm} to obtain fibrations
\begin{equation}\label{splitfibz}
\varOmega\zk\longrightarrow\varOmega\djs(K)\longrightarrow T^V
\sands
\varOmega M\longrightarrow\varOmega\djs(K)
\stackrel{\varOmega l}{\llongrightarrow}T^n.
\end{equation}
Each of these admits a section, defined by the $m$ generators of
$\pi_2(\djs(K))\cong\Z^V$ and by the duals of $l_1$, \dots, $l_n$
respectively, and are therefore split in $\cat{top}$. So we have
homotopy equivalences
\[
\varOmega\djs(K)\stackrel{\simeq}{\longrightarrow}\varOmega\zk\times
T^V
\sands
\varOmega\djs(K)\stackrel{\simeq}{\longrightarrow}\varOmega M\times
T^n,
\]
which do \emph{not} preserve $H$-space structures. We therefore obtain
exact sequences of Pontrjagin algebras
\begin{equation}\label{paseq}
0\longrightarrow H_*(\varOmega S;R)\longrightarrow
H_*(\varOmega\djs(K);R)\longrightarrow\wedge(W_S)\longrightarrow 0,
\end{equation}
where $W_S=U$ or $n$ as $S=\zk$ or $M$; these do not usually split as
algebras.

Our models rely heavily on the loop and classifying functors
$\varOmega_*$ and $B_*$ of Subsection \ref{adpa}. For any
simply-connected CW-complex $X$, the reduced singular chain complex
$C_*(X;R)$ is an object of $\cat{dgc}$ under the Alexander-Whitney
diagonal, and Adams \cite{adam56} provides a chain equivalence
$\varOmega_*C_*(X;R)\to C_*(\varOmega X;R)$.

When $R$ is $\Q$, the source of the adjunction \eqref{barcobar} is a
full subcategory of a model category, and its target is model.
\begin{proposition}\label{bcqpa}
The loop functor $\varOmega_*$ preserves cofibrations of connected
coalgebras and weak equivalences of simply connected coalgebras; the
classifying functor $B_*$ preserves fibrations of connected algebras
and all weak equivalences.
\end{proposition}
\begin{proof}
The fact that $B_*$ and $\varOmega_*$ preserve weak equivalences of
algebras and simply connected coalgebras respectively is
proved by standard arguments with the Eilenberg-Moore spectral
sequence \cite[page~538]{f-h-t92}. The additional assumption for
coalgebras is necessary to ensure that the cobar spectral sequence
converges, because the relevant filtration is decreasing.

Given any cofibration $i\colon C_1\to C_2$ of connected coalgebras,
we must check that $\varOmega_*i\colon\varOmega_* C_1\to\varOmega_* C_2$
satisfies the left lifting property with respect to any acyclic
fibration $p\colon A_1\to A_2$ in $\cat{dga}_R$. This involves
finding lifts $\varOmega_*C_2\to A_1$ and $C_2\to B_*A_1$ in the
respective diagrams
\[
\begin{CD}
  \varOmega_*C_1 @>>> A_1\\
  @V{\varOmega_*i}VV @VVpV\\
  \varOmega_*C_2 @>>> A_2
\end{CD}\spandsp
\begin{CD}
  C_1 @>>> B_*A_1\\
  @ViVV @VV{B_*p}V\\
  C_2 @>>> B_*A_2
\end{CD}\quad\quad;
\]
each lift implies the other, by adjointness. Since $p$ is an acyclic
fibration, its kernel $A$ satisfies $H(A)\cong \Q$. Moreover, the
projection $B_*p$ splits by~\cite[Theorem~IV.2.5]{h-m-s74}, so
$B_*A_1$ is isomorphic to the cofree product $B_*A_2\star B_*A$. In
this case $B_*p$ is an acyclic fibration in $\cat{dgc}_{0,R}$, and our
lift is assured.

A second application of adjointness shows that $B_*$ preserves all
fibrations of connected algebras.
\end{proof}

\begin{remarks}
It follows from Proposition \ref{bcqpa} that the restriction
of~\eqref{barcobar} to
\[
\varOmega_*\colon\cat{dgc}_{1,R}\rla\cat{dga}_{0,R}:\! B_*
\]
acts as a Quillen pair, and induces an adjoint pair of equivalences on
appropriate full subcategories of the homotopy categories. An example
is given in~\cite[p.~538]{f-h-t92} which shows that $\varOmega_*$
fails to preserve quasi-isomorphisms (or even acyclic cofibrations) if
the coalgebras are not simply connected.
\end{remarks}

We now focus on the situation when $X$ is $\djs(K)$. By dualising the
formality results of \cite[Theorem 4.8]{no-ra05}, we obtain a
zig-zag of quasi-isomorphisms
\begin{equation}\label{cdgam}
  C_*(\djs(K);\Q)\lzze Q\br{K}.
\end{equation}
in $\cat{dgc}$. Since $\varOmega_*$ preserves quasi-isomorphisms,
\eqref{cdgam} combines with Adams's results \cite{adam56} to determine
an isomorphism
\begin{equation}\label{loopcotor}
 H(\varOmega_*\Q\br{K},d)\;\letbe\;
 \mathop{\mathrm{Cotor}}\nolimits_{\Q\langle K\rangle}(\Q,\Q)
 \stackrel{\cong}{\longrightarrow}H_*(\varOmega\djs(K);\Q)
\end{equation}
of graded algebras. Our first model for $\varOmega\djs(K)$ is
therefore $\varOmega_*\Q\br{K}$ in $\cat{dga}$.

The graded algebra underlying $\varOmega_*R\br{K}$ is the tensor
algebra $T(s^{-1}\,\overline{\!R}\br{K})$ on the desuspended
$R$-module $\,\overline{\!R}\br{K}$; the differential is defined on
generators by
\[
  d(s^{-1}v\br{\s})=
  \sum_{\s=\tau\sqcup\tau';\;\tau,\,\tau'\ne\varnothing}
  s^{-1}v\br{\tau}\otimes s^{-1}v\br{\tau'},
\]
because $d=0$ on $R\br{K}$. For future purposes it is convenient
to write $s^{-1}v{\br\s}$ as $\chi_\s$ for any face $\sigma\in K$, and
to use the abbreviations
\begin{equation}\label{defchis}
\chi_i\letbe\chi_{v_i}\sands\chi_{ij}\letbe\chi_{\{v_i,v_j\}}\quad
\text{for all}\quad 1\leq i,j\leq m.
\end{equation}

Our aim is to construct models that satisfy algebraic analogues
of~\cite[Theorem~7.17]{p-r-v04}. This asserts the existence of a
commutative diagram
\begin{equation}\label{prvdiag}
\begin{CD}
  \varOmega\hocolim^{\scat{top}}BT^K @>\overline{h}_K>>
  \hocolim^{\scat{tmon}}T^K\\ @VV\varOmega p_KV @VVV\\
  \varOmega\djs(K) @>h_K>> \colim^{\scat{tmon}}T^K
\end{CD},
\end{equation}
in $\Ho(\cat{tmon})$, where $\varOmega p_K$ and $\overline{h}_K$ are
homotopy equivalences for any $K$, and $h_K$ is a homotopy equivalence
when $K$ is flag. We interpret $T^K$ as a diagram of topological
monoids, rather than of topological spaces as in Proposition
\ref{firstinvt}.

Our previous algebraic models for $T^K$ have been commutative,
contravariant, and cohomological, but to investigate \eqref{prvdiag}
we introduce models that are covariant and homological. They involve
the diagrams
\[
\wedge^K(U)\colon\cat{cat}(K)\to\cat{dga}\sands
\cl^K(U)\colon\cat{cat}(K)\to\cat{dgl},
\]
which assign to $\s\subseteq\tau$ the monomorphisms
$\wedge(\s)\to\wedge(\tau)$ and $\cl(\s)\to\cl(\tau)$, of exterior
algebras and commutative Lie algebras respectively, on
$1$--dimensional generators with zero differentials. Thus
$\wedge^K(U)$ may be identified with the diagram $H_*(T^K;\Q)$ of
Pontrjagin rings, and $\cl^K(U)$ with its diagram
$\pi_*(T^K)\otimes_\Z\Q$ of primitives. For $BT^K$, we consider the
diagram $S^K\br{V}\colon\cat{cat}(K)\to\cat{dgc}$, which assigns to
$\s\subseteq\tau$ the monomorphism $S\br{\s}\to S\br{\tau}$ of
coalgebras, on $2$--dimensional generators with zero differential. It
may be identified with the diagram $H_*(BT^K;\Q)$ of homology
coalgebras.

The individual algebras and coalgebras in these diagrams are all
symmetric, but the context demands they be interpreted in the
non-commutative categories; this is especially important when forming
limits and colimits.

\begin{proposition}\label{symwe}
\sloppy There are acyclic fibrations
$\varOmega_*S\br{V}\to\wedge(U)$ in $\cat{dga}$ and
$L_*S\br{V}\to\cl(U)$ in $\cat{dgl}$, for any set of vertices $V$.
\end{proposition}

\begin{proof}
  Using \eqref{defchis}, we define the first map by $\chi_i\mapsto
  u_i$ for $1\le i\le m$. Because
\begin{equation}\label{cobarrel}
  d\chi_{ii}=\chi_i\otimes\chi_i \sands
  d\chi_{ij}=\chi_i\otimes\chi_j+\chi_j\otimes\chi_i\quad
  \text{for}\quad i\ne j
\end{equation}
hold in $\varOmega_*S\br{V}$, the map is consistent with the exterior
relations in its target. So it is an epimorphism and quasi-isomorphism
in $\cat{dga}$, and hence an acyclic fibration. The corresponding
result for $\cat{dgl}$ follows by restriction to primitives.
\end{proof}

By Proposition \ref{symwe} there are acyclic fibrations
$\varOmega_*S\br{\s}\to\wedge(\sigma)$ for every face $\s\in K$. These
induce an acyclic Reedy fibration
$e\colon\varOmega_*S^K\br{V}\to\wedge^K(U)$, and therefore a
commutative diagram
\begin{equation}\label{cda}
\begin{CD}
\hocolim^{\scat{dga}}\varOmega_*S^K\br{V}@>\simeq>>
\hocolim^{\scat{dga}}\wedge^K(U)\\
@VV\simeq V@VVV\\
\colim^{\scat{dga}}\varOmega_*S^K\br{V}@>>>
\colim^{\scat{dga}}\wedge^K(U)
\end{CD}
\end{equation}
in $\cat{dga}$. The left-hand arrow is a quasi-isomorphism by
Proposition \ref{hotolim}, because $\varOmega_*S^K\br{V}$ is Reedy
cofibrant; the upper arrow is a quasi-isomorphism by Remark
\ref{htpyinvce}, because $e$ is an equivalence.

We may now introduce our first algebraic model for diagram
\eqref{prvdiag}.
\begin{theorem}\label{hcldi}
There is a commutative diagram
\[
\begin{CD}
  \varOmega_*\hocolim^{\scat{dgc}}S_K\br{V} @>\overline{\eta}_K>>
  \hocolim^{\scat{dga}}\wedge^K(U)\\
  @VV\varOmega_*\rho_KV @VVV\\
  \varOmega_*\Q\langle K\rangle @>\eta_K>> \colim^{\scat{dga}}\wedge^K(U)
\end{CD},
\]
in $\Ho(\cat{dga})$, where $\varOmega_*\rho_K$ and $\overline{\eta}_K$
are isomorphisms for any $K$.
\end{theorem}
\begin{proof}
Consider the diagram
\begin{equation}\label{cdb}
\varOmega_*\hocolim^{\scat{dgc}}S^K\br{V}\stackrel{\simeq}{\longrightarrow}
\varOmega_*\colim^{\scat{dgc}}S^K\br{V}\stackrel{\cong}{\longleftarrow}
\colim^{\scat{dga}}\varOmega_*S^K\br{V}
\end{equation}
in $\cat{dga}$. The first arrow is a quasi-isomorphism by Proposition
\ref{hotolim}, because $S^K\br{V}$ is Reedy cofibrant in $\cat{dgc}$
and $\varOmega_*$ preserves quasi-isomorphisms; the second arrow is an
isomorphism because $\varOmega_*$ is a left adjoint. So we may combine
\eqref{cdb} with the upper left hand corner of \eqref{cda} to obtain a
zig-zag
\[
\varOmega_*\hocolim^{\scat{dgc}}S^K\br{V}
\lzze\hocolim^{\scat{dga}}\varOmega_*S^K\br{V}
\stackrel{\simeq}{\longrightarrow}\hocolim^{\scat{dga}}\wedge^K(U)
\]
of weak equivalences in $\cat{dga}$, which we label
$\overline{\eta}_K$. The result follows by amalgamating diagrams
\eqref{cda} and \eqref{cdb}, rewriting $\colim^{\scat{dgc}}S^K\br{V}$
as $\Q\br{K}$ throughout, and passing to the homotopy category.
\end{proof}
By construction, $\eta_K$ is the fibration
\begin{equation}\label{etaK}
\varOmega_*\Q\br{K}\cong\varOmega_*\colim^{\scat{dgc}}S^K\br{V}
\to\colim^{\scat{dga}}\wedge^K(U)
\end{equation}
in $\cat{dga}$. Corollary~\ref{ldjcolim} and Theorem~\ref{djcoform} show
that $\eta_K$ is acyclic when $K$ is flag.

The following statement is proved similarly.
\begin{theorem}\label{liedi}
There is a commutative diagram
\[
\begin{CD}
  L_*\hocolim^{\scat{cdgc}}S_K\br{V} @>\overline{\lambda}_K>>
  \hocolim^{\scat{dgl}}\cl^K(U)\\
  @VVL_*\rho_KV @VVV\\
  L_*\Q\langle K\rangle @>\lambda_K>> \colim^{\scat{dgl}}\cl^K(U)
\end{CD},
\]
in $\Ho(\cat{dgl})$, where $L_*\rho_K$ and $\overline{\lambda}_K$ are
isomorphisms for any $K$.
\end{theorem}
Theorem~\ref{djcoform} confirms that $\lambda_K$ is a quasi-isomorphism
when $K$ is flag.
\begin{corollary}
For any simplicial complex $K$, there are isomorphisms
\begin{align*}
  H_*(\varOmega\djs(K);\Q)&\cong
  H\bigl(\hocolim^{\scat{dga}}\wedge^K(U)\bigr)\\
  \pi_*(\varOmega\djs(K))\otimes_\Z\Q&\cong
  H\bigl(\hocolim^{\scat{dgl}}\cl^K(U)\bigr)
\end{align*}
of graded algebras and Lie algebras respectively.
\end{corollary}

%
%
%
%
%
%
%
%
%

\section{Models for loop spaces: flag complexes $K$}\label{mlsfk}

In this section we study the loop spaces associated to \emph{flag
  complexes} $K$. Such complexes have significantly simpler
combinatorial properties, which are reflected in the homotopy theory
of the toric spaces. We modify results of the previous section in this
context, and focus on applications to the rational Pontrjagin rings
and homotopy Lie algebras of $\varOmega\djs(K)$ and $\varOmega\zk$.

For any simplicial complex $K$, a subset $\zeta\subseteq V$ is called
a \emph{missing face} when every proper subset lies in $K$, but
$\zeta$ itself does not. If every missing face of $K$ has $2$
vertices, then $K$ is a flag complex; equivalently, $K$ is flag when
every set of vertices that is pairwise connected spans a simplex. A
flag complex is therefore determined by its 1-skeleton, which is a
graph. When $K$ is flag, we may express the Stanley-Reisner algebra as
\begin{equation}\label{srflag}
  R[K]=T_R(V)\bigr/
  (v_iv_j-v_jv_i=0\text{ for }\{i,j\}\in K,\;
  v_iv_j=0\text{ for }\{i,j\}\notin K)
\end{equation}
over any ring $R$. It is therefore \emph{quadratic}, in the sense
that it is the quotient of a free algebra by quadratic relations.

The following result of Fr\"oberg~\cite[\S3]{frob75} allows us to
calculate the Yoneda algebras $\Ext_A(\Q,\Q)$ explicitly for a class
of quadratic algebras $A$ that includes Stanley-Reisner algebras of
flag complexes.
\begin{proposition}\label{srkos}
As graded algebras, $\Ext_{\Q[K]}(\Q,\Q)$ is isomorphic to
\begin{equation}\label{srext}
  T(U)\bigr/(u_i^2=0,\; u_iu_j+u_ju_i=0\text{ for }\{i,j\}\in K)
\end{equation}
for any flag complex $K$.
\end{proposition}

\begin{remark}
  Algebra~\eqref{srext} is the \emph{quadratic dual} of
  \eqref{srflag}. A quadratic algebra $A$ is called
  \emph{Koszul}~\cite{sh-yu97} if its quadratic dual coincides with
  $\Ext_A(\Q,\Q)$, so Proposition \ref{srkos} asserts that $\Q[K]$ is
  Koszul whenever $K$ is flag.
\end{remark}

\begin{theorem}\label{hldj}
For any flag complex $K$, there are isomorphisms
\begin{align*}
H_*(\varOmega\djs(K);\Q)&\;\cong\;
T(U)\bigr/
  (u_i^2=0,\; u_iu_j+u_ju_i=0\text{ for }\{i,j\}\in K)\\
\pi_*(\varOmega\djs(K))\otimes_\Z\Q&\;\cong\;
\fl(U)\bigr/
  \bigl([u_i,u_i]=0,\; [u_i,u_j]=0\text{ for }\{i,j\}\in K\bigr).
\end{align*}
\end{theorem}
\begin{proof}
  The first isomorphism combines \eqref{cotorext} and
  \eqref{loopcotor} with Proposition~\ref{srkos}, and the second
  follows by restriction to primitives.
\end{proof}
\begin{corollary}\label{ldjcolim}
  The algebras $H_*(\varOmega\djs(K);\Q)$ and
  $\pi_*(\varOmega\djs(K))$ are isomorphic to
  $\colim^{\scat{dga}}\wedge^K(U)$ and $\colim^{\scat{dgl}}\!\cl^K(U)$
  respectively, whenever $K$ is flag.
\end{corollary}
\begin{proof}
  The algebras of Theorem \ref{hldj} are the colimits of the stated
  diagrams, because $\colim^{\scat{dga}}\wedge^K(U)$ is formed in the
  non-commutative setting of $\cat{dga}$.
\end{proof}

The \emph{Poincar\'e series} $F(\Q[K];t)$ is computed
in~\cite[Theorem.~II.1.4]{stan96} for any complex $K$, although we
find it convenient to adopt the re-grading associated to
$1$--dimensional vertices $u_i$. If $K$ is $(n-1)$--dimensional and
has $f_i$ faces of dimension $i$ for $i\geq0$, the series becomes
\begin{equation}\label{psfr}
  \sum_{i=-1}^{n-1}\frac{f_it^{(i+1)}}{(1-t)^{i+1}}\;=\;
\frac{h_0+h_1t+\ldots+h_nt^{n}}{(1-t)^n}\;,
\end{equation}
where $f_{-1}=1$ and the integers $h_i$ are defined by \eqref{psfr}
for $0\le i\le n$. The integral vectors $(f_{-1},\ldots,f_{n-1})$ and
$(h_0,\ldots,h_n)$ are known to combinatorialists as the
\emph{$f$-vector} and \emph{$h$-vector} of $K$ respectively.

\begin{proposition}\label{loopsps}
For any flag complex $K$, we have that
\[
  F\bigl( H_*(\varOmega\djs(K);\Q);t \bigr)\;=\;
  \frac{(1+t)^n}{1-h_1t+\ldots+(-1)^nh_nt^n}\;.
\]
\end{proposition}
\begin{proof}
  Since $H_*(\varOmega\djs(K);\Q)$ is the quadratic dual of $\Q[K]$,
  the identity
\[
  F\bigl(\Q[K];-t\bigr)\cdot F\bigl(H_*(\varOmega\djs(K));t\bigr)\;=\;1
\]
follows from Fr\"oberg~\cite[\S4]{frob75}. When applied to
\eqref{psfr}, it yields our result.
\end{proof}

The Poincar\'e series of $\pi_*(\varOmega\djs(K))\otimes_\Z\Q$ is
calculated in~\cite[\S4.2]{de-su06}.

We conclude by examining the coformality $\djs(K)$, and applying the
results to Theorems~\ref{hcldi} and \ref{liedi}.

\begin{theorem}\label{djcoform}
The space $\djs(K)$ is coformal if and only if $K$ is flag.
\end{theorem}
\begin{proof}
  If $K$ is flag, we may compose \eqref{etaK} with the first
  isomorphism of Corollary~\ref{ldjcolim} to obtain an epimorphism
  $\varOmega_*\Q\br{K}\to H_*(\varOmega\djs(K);\Q)$ in $\cat{dga}$; it
  is a quasi-isomorphism because both algebras have the same homology.
  Restricting to primitives yields a quasi-isomorphism $e\colon
  L_*(\Q\br{K})\rightarrow\pi_*(\varOmega\djs(K))\otimes_\Z\Q$ in
  $\cat{dgl}$, by definition of $L_*$.

  Now choose a minimal model $M_K\to\Q[K]$ for the Stanley-Reisner
  algebra in $\cat{cdga}$. Its graded dual $\Q\br{K}\to C_K$ is a
  minimal model for $\Q\br{K}$ \cite[\S 5]{ne-mi78} in $\cat{cdgc}$,
  so $\varOmega_*\Q\br{K}\to\varOmega_*C_K$ is a weak equivalence in
  $\cat{dga}$. Restricting to primitives provides the the central map
  in the zig-zag
\begin{equation}\label{cform}
  L_K\stackrel{\simeq}{\longrightarrow}L_*C_K
  \stackrel{\simeq\!}{\longleftarrow}L_*\Q\br{K}
  \stackrel{e}{\longrightarrow}\pi_*(\varOmega\djs(K))\otimes_\Z\Q
\end{equation}
of quasi-isomorphisms in $\cat{dgl}$, where $L_K$ is a minimal model
for $\djs(K)$ in $\cat{dgl}$ \cite[\S 8]{ne-mi78}. Hence $\djs(K)$ is
coformal.

On the other hand, every missing face of $K$ with $>2$ vertices
determines a non-trivial higher Samelson bracket in
$\pi_*(\varOmega\djs(K))\otimes_\Z\Q$ (see Example~\ref{accex}(2)).
The existence of such brackets in $\pi_*(\varOmega X)\otimes_\Z\Q$
ensures that $X$ cannot be coformal, just as higher Massey products in
$H^*(X;\Q)$ obstruct formality.
\end{proof}

%
%
%
%
%
%
%
%
%

\section{Examples}\label{examp}

In our final section we review examples that illustrate calculations
with the cobar construction, and clarify the importance of higher
brackets.  The examples refer to our results for flag complexes, and
offer clues to the structure of Pontrjagin rings and rational homotopy
Lie algebras for more general~$K$.

\begin{examples}\label{ex1}\hfill{}

(1) Let $K$ be the simplex $\Delta^{n-1}$, so that $h_0=1$ and
$h_i=0$ for $i>0$.  Then $\djs(K)$ is homeomorphic to $(\C
P^\infty)^n$, and $\varOmega\djs(K)$ is homotopy equivalent to
$T^n$; also $F( H_*(\varOmega\djs(K));t )\;=\;(1+t)^n$, in
accordance with Corollary~\ref{loopsps}.

(2) Let $K$ be the boundary $\partial\Delta^n$, so that
$h_0=\dots=h_n=1$. Then $\varOmega\djs(K)$ is homotopy equivalent to
$\varOmega S^{2n+1}\times T^{n+1}$, and
\[
  F\bigl( H_*(\varOmega\djs(K));t \bigr)\;=\;
  \frac{(1+t)^{n+1}}{1-t^{2n}}.
\]
On the other hand, Corollary~\ref{loopsps} gives
\[
  \frac{(1+t)^n}{1-t+t^2+\ldots+(-1)^nt^n}\;=\;
\frac{(1+t)^{n+1}}{1+(-1)^nt^{n+1}}.
\]
The formulae agree if $n=1$, in which case $K$ is flag, but differ
otherwise.
\end{examples}

\begin{examples}\label{accex}\hfill{}

(1) Let $K$ be a discrete complex on $m$ vertices. Over any ring
$R$, the cobar construction $\varOmega_*R\langle K\rangle$ on the
corresponding Stanley-Reisner coalgebra is generated as an algebra
by the elements of the form $\chi_{i\ldots i}$ with $i\in[m]$. The
first identity of~\eqref{cobarrel} still holds, but
$\chi_i\otimes\chi_j+\chi_j\otimes\chi_i$ is no longer a
coboundary for $i\ne j$ since there is no element $\chi_{ij}$ in
$\varOmega_*R\langle K\rangle$. We obtain a quasi-isomorphism
\[
\varOmega_*R\langle K\rangle\longrightarrow
T_R(u_1,\ldots,u_m)/(u_i^2=0,\; 1\le i\le m)
\]
that maps $\chi_i$ to the homology class $u_i$. The right hand side is
isomorphic to $H_*(\varOmega\djs(K);R)$, in accordance with
Theorem~\ref{hldj}.

The space $\djs(K)$ is a wedge of $m$ copies of $\C P^\infty$. In
this case, $\zk$ is the homotopy fibre of the inclusion $\bigvee^m\C
P^\infty\to(\C P^\infty)^m$ of axes, and may be identified
with the complement
\begin{equation}\label{wedge}
   \C^m\setminus\bigcup_{1\le i<j\le m}\{z_i=z_j=0\}
\end{equation}
of the codimension-two coordinate subspaces
(see~\cite[Example~8.15]{bu-pa02}). If $m=3$, its homotopy type may be
identified by desuspending as follows:
\begin{multline*}
  \C^3\setminus\bigcup_{1\le i<j\le 3}\{z_i=z_j=0\}
  \;\simeq\; S^5\setminus(S^1\cup S^1\cup S^1)\;\simeq\;\Sigma^3
  \bigl(S^2\setminus(S^1\cup S^1\cup S^1)\bigr)\\
  \;\simeq\;\Sigma^3(S^0\vee S^0\vee S^0\vee
  S^1\vee S^1)\;\simeq\; S^3\vee S^3\vee S^3\vee
  S^4\vee S^4,
\end{multline*}
where the three circles are disjoint in $S^5$. By~\cite{gr-th04}, the
complement~\eqref{wedge} is homotopy equivalent to a wedge of spheres
for all $m$.

The loop space $\varOmega\djs(K)$ is homotopy equivalent to a free
product \mbox{$T*\dots*T$} of $m$ circles~\cite[Example~6.8]{p-r-v04},
and the above calculation shows that its homology is a free product of
$m$ exterior algebras on a single generator of degree $1$. Also,
$\varOmega\zk$ is homotopy equivalent to the commutator subgroup of
$T*\dots*T$.

(2) Consider the simplest non-flag complex $K=\partial\Delta^2$,
of Example \ref{ex1}(2). Apart from the elements $\chi_i$ and
their products in $\varOmega_*R\langle\partial\Delta^2\rangle$,
there is an additional 4-dimensional cycle
\[
  \psi:=\chi_1\chi_{23}+\chi_2\chi_{13}+\chi_3\chi_{12}+
  \chi_{12}\chi_3+\chi_{13}\chi_2+\chi_{23}\chi_1,
\]
whose failure to bound is due to the non-existence of
$\chi_{123}$. Relations~\eqref{cobarrel} hold, and give rise to
the exterior relations between the corresponding homology classes
$u_1,u_2,u_3$. Moreover, a direct check shows that the elements
$\chi_i\psi-\psi\chi_i$ are boundaries for $1\le i\le 3$. We
therefore have
\[
  H(\varOmega_*R\langle\partial\Delta^2\rangle)\;=\;
  H_*(\varOmega\djs(\partial\Delta^2);R)\;\cong\;
  \wedge(u_1,u_2,u_3)\otimes S(w),
\]
where $w=[\psi]$, \ $\deg u_i=1$, \ $\deg w=4$. This calculation
generalises easily to the case $K=\partial\Delta^{m-1}$ with
$m\ge3$, and we obtain similarly
\begin{equation}\label{seew}
  H_*(\varOmega\djs(\partial\Delta^{m-1});R)\;\cong\;
  \wedge(u_1,\ldots,u_m)\otimes S(w),
\end{equation}
where $w$ is the homology class of the $(2m-2)$-cycle $\psi\in
\varOmega_*R\langle\partial\Delta^{m-1}\rangle$, whose failure to
bound is due to the non-existence of $\chi_{1\ldots m}$.
The exact sequence~\eqref{paseq} takes the form
\[
  0\longrightarrow S(w)\longrightarrow
H_*(\varOmega\djs(\partial\Delta^{m-1});\Z)\longrightarrow\wedge(U)
\longrightarrow 0,
\]
where $\deg w=2m-2$. It follows that this sequence splits for $m\ge3$,
but not for $m=2$ as the previous calculation shows.

Unlike the situation in the previous example, there is no
quasi-isomorphism $\varOmega_*R\langle\partial\Delta^{m-1}\rangle\to
H(\varOmega_* R\langle\partial\Delta^{m-1}\rangle)$ for $m\ge3$. The
element $w$ of \eqref{seew} is the \emph{higher commutator product};
it is the Hurewicz image of the \emph{higher Samelson product} of
$1$--dimensional generators $u_i\in\pi_1(\varOmega\djs(K))$ for $1\le
i\le m$, see~\cite{will72}. It reduces to the ordinary commutator for
$m=2$. The fact that $w$ is the non-trivial higher commutator product
of $u_1,\ldots,u_m$ constitutes the additional homology information
necessary to distinguish between the topological monoids
$\varOmega\djs(\partial\Delta^{m-1})$ and $\varOmega S^{2m-1}\times
T^V$.

(3) We finish by considering a more complicated non-flag complex,
namely the $1$--skeleton of $\Delta^3$. Arguments similar to those
of Example (2) show that the Pontrjagin ring
$H_*(\varOmega\djs(K);R)$ is multiplicatively generated by four
$1$--dimensional classes $u_1,\ldots,u_4$ and four
$4$--dimensional classes $w_{123}$, $w_{124}$, $w_{134}$,
$w_{234}$, corresponding to the four missing faces with three
vertices each.
Identities~\eqref{cobarrel} give rise to the exterior relations
between $u_1,\ldots,u_4$. We may easily check that $u_i$ commutes
with $w_{jkl}$ if $i\in\{j,k,l\}$. The remaining non-trivial
commutators are subject to one extra relation, which can be
derived as follows. Consider the relation
\begin{equation}\label{reln}
  d\chi_{1234}\;=\;(\chi_1\chi_{234}+\chi_{234}\chi_1)
  +\ldots+(\chi_4\chi_{123}+\chi_{123}\chi_4)+\beta
\end{equation}
in $\varOmega_*S_R\br{v_1,v_2,v_3,v_4}$, where $\beta$ consists of
terms $\chi_{\sigma}\chi_{\tau}$ such that $|\sigma|=|\tau|=2$.
Denote the first four summands in the right hand side of
\eqref{reln} by $\alpha_1$, $\alpha_2$, $\alpha_3$, $\alpha_4$
respectively, and apply the differential to both sides. Observing
that
$d\alpha_1=-\chi_1\psi_{234}+\psi_{234}\chi_1=-[\chi_1,\psi_{234}]$,
and similarly for $d\alpha_2$, $d\alpha_3$ and $d\alpha_4$, we
obtain
\[
  [\chi_1,\psi_{234}]+[\chi_2,\psi_{134}]+
  [\chi_3,\psi_{124}]+[\chi_4,\psi_{123}]\;=\;d\beta.
\]
The outcome is an isomorphism
\[
  H_*(\varOmega\djs(K);R)\;\cong\;
  T_R(u_1,u_2,u_3,u_4,w_{123},w_{124},w_{134},w_{123})/I,
\]
where $\deg w_{ijk}=4$ and $I$ is generated by three types of
relation:
\begin{enumerate}
\item exterior algebra relations for $u_1,u_2,u_3,u_4$;
\item $[u_i,w_{jkl}]=0$ for $i\in\{j,k,l\}$;
\item
  $[u_1,w_{234}]+[u_2,w_{134}]+[u_3,w_{124}]+[u_4,w_{123}]=0$.
\end{enumerate}
As $w_{ijk}$ is the higher commutator of $u_i$, $u_j$ and $u_k$,
(3) may be considered as a higher analogue of the Jacobi identity.
\end{examples}

\bibliographystyle{amsalpha}

\end{document}